\documentclass[a4paper,11pt]{article}
\usepackage{bbm}
\usepackage{}
\usepackage{mathrsfs}
\usepackage{amsfonts}
\usepackage{amsfonts}
\usepackage{amssymb}
\usepackage{mathrsfs}
\usepackage{indentfirst,latexsym,bm,amsmath,amssymb,amsthm}
\usepackage{xcolor}
\usepackage[
                   bookmarksnumbered=true,
                   bookmarksopen=true, 
                   pdfauthor=kmc,
                   pdfcreator={LaTex with hyperref package + WinEdt},
                   pdftitle=trudge,
                   colorlinks,%
                   citecolor=blue,
                   linkcolor=blue,%
                   urlcolor=blue,
                   hyperindex,%
                   plainpages=false,%
                   pdfstartview=FitH,
                   linktocpage=true,
                   dvipdfm%
                   ]{hyperref}
\usepackage{indentfirst,latexsym,bm,amsmath,amssymb,amsthm}
\usepackage[dvips]{graphicx}

\renewcommand{\theequation}{\thesection.\arabic{equation}}
\makeatletter\@addtoreset{equation}{section} \makeatother
\newtheorem{thm}{Theorem}[section]
\newtheorem{Lemma}{Lemma}[section]
\newtheorem{proposition}{Proposition}[section]

\newtheorem{rem}{Remark}[section]

 

\title{Global low regularity solutions for nonlinear elastic waves}

\author{Kunio Hidano\thanks{Department of Mathematics, Faculty of Education, Mie University,
1577 Kurima-machiya-cho, Tsu, Mie 514-8507, Japan.
{ E-mail address: hidano@edu.mie-u.ac.jp}} ~~~~  Dongbing Zha\thanks{  Corresponding author. Department of Applied Mathematics, Donghua University, Shanghai 201620, PR China.{ E-mail address: ZhaDongbing@163.com}     }}

\begin{document}

\maketitle

\begin{abstract}
We study the Cauchy problem for 3-D nonlinear elastic waves satisfying the null condition with
low regularity initial data. In the radially symmetric case,
we prove the global existence of a low regularity solution for every small data in $H^3\times H^2$ with a low weight.  \\
\emph{keywords}: Elastic waves; global low regularity solutions; radial symmetry. 
\end{abstract}

\pagestyle{plain} \pagenumbering{arabic}

\section{Introduction}
For elastic materials, the motion for the displacement is governed by the nonlinear elastic wave equation which is a second-order quasilinear hyperbolic system. For isotropic, homogeneous, hyperelastic materials, the motion for the displacement $u=u(t,x)$ satisfies
\begin{align}\label{elastic}
\partial_t^2u-c_2^2\Delta u-(c_1^2-c_2^2)\nabla\nabla\cdot u=F(\nabla u,\nabla^2u),
 \end{align}
 where the nonlinear term $F(\nabla u,\nabla^2u)$ is linear in $\nabla^2u$ and will be described explicitly in later sections. Some physical backgrounds of nonlinear elastic waves can be found in
 Ciarlet \cite{MR936420} and Gurtin \cite{Gurtin81}.
 \par
 The short time existence of classical solutions for nonlinear elastic waves is standard now \cite{MR0420024}.
 The research for long time existence of classical solutions for nonlinear elastic waves started from Fritz John's pioneering work on elastodynamics (see~Klainerman \cite{Klainerman98}). In the 3-D case, John showed that local classical solutions in general will develop singularities for radial and small initial data \cite{John84}, and they almost globally exist for small data with compact support \cite{John88}. See also simplified proof of almost global existence in  \cite{Klainerman96}, \cite{MR3650329} and a lower bound estimate in \cite{MR3014806}.
In order to ensure the global existence of classical solutions with small initial data, some structural condition on the nonlinearity which is called the null condition is necessary. We refer the reader to Agemi \cite{Agemi00} and~Sideris \cite{Sideris00} (see also a previous result in Sideris \cite{Sideris96}).
The exterior domain analogues of John's almost global existence result and Agemi and~Sideris's global existence result were obtained in \cite{MR2249996} and \cite{MR2344575}, respectively. In the 2-D case, under the null condition global small classical solution with radial initial data was constructed in \cite{zhaelastic} and \cite{zha3} for the Cauchy problem and the exterior domain problem, respectively. The non-radially symmetric case is still open. \par
    All the above works are in the framework of classical solutions. Studies on local low regularity well-posedness for the Cauchy problem of nonlinear
wave equations started from Klainerman-Machedon's pioneering work \cite{MR1231427} in the semilinear case. For the general quasilinear equations, we refer the reader to the sharpest result \cite{MR2178963}, \cite{MR3656947} and the references therein. On the other hand, Lindblad constructed
some counterexamples to show sharp results of ill-posedness \cite{MR1375301}, \cite{MR1666844}. For the Einstein equation in wave coordinates, the sharpest
result of well-posedness and ill-posedness was shown by Klainerman and Rodnianski \cite{MR2180400} and Ettinger and Lindblad \cite{MR3583356}, respectively.\par
  We note that a family of special solutions which are obtained via the
Lorentz transformation has often played a key role in the proof of ill-posedness.
It is therefore natural to expect that we can improve regularity requirements
for initial data if we rule out such special solutions. It is indeed the case. It should be mentioned that
improvement of regularity in the presence of
radial symmetry was first observed in \cite{MR1231427}
for semilinear equations such as
$\Box u=c_1(\partial_t u)^2+c_2|\nabla u|^2$.
(See also page 176 of \cite{MR1211729},
Section 5 of \cite{MR2108356}, and
 \cite{MR2262091},
 \cite{MR2971205}, and
\cite{Hidano17} for related results.)
We also mention that
it was also observed for the wave equation with
a power-type nonlinear term $\Box u=F(u)$,
first by Lindbald and Sogge \cite{MR1335386}, and then by
some authors (see  \cite{MR2569550},
 \cite{MR2911108},  \cite{MR3286047}), and quasilinear wave equations of the form
$\partial_t^2u-a^2(u)\Delta u=c_1(\partial_t u)^2+c_2|\nabla u|^2$, by \cite{MR2919103} and  \cite{MR2482537}.
Particularly,
the almost global existence of low regularity radially symmetric solutions with small initial data was showed in \cite{MR2262091} and \cite{MR2919103} for the semilinear and the quasilinear case, respectively. A global
regularity result can be found in \cite{MR2482537}, which is the first result for global existence of low regularity
solutions to quasilinear wave equations.
 \par

In the present paper, we will study global low regularity solutions for the Cauchy problem of 3-D nonlinear elastic waves with the null condition,
and focus on the radially symmetric case. For this purpose,
we first note the fact that in the radially symmetric case, the nonlinear elastic wave equation can be reduced to a system of quasilinear wave equations with single wave speed. To analyze this system, we will use the ghost weight energy estimate of Alinhac to display the null condition and get enough decay in time on the region $r\geq \langle t\rangle/2$, and employ the Klainerman-Sideris type estimate and a refined Keel-Smith-Sogge (KSS, for short) type estimate to control the elastic waves with rather low regularity on the region $1/2\leq r\leq \langle t\rangle/2$ and $r\leq 1/2$. Together with some weighted Sobolev inequalities for radially symmetric functions, we can get
global existence for nonlinear elastic waves
satisfying the null condition
when radial data is small with respect to
the $H^3\times H^2$ norm with a low weight $\langle x\rangle$. Note that even the standard local existence theorem for nonlinear elastic waves (see \cite{MR0420024}) need the regularity requirement $H^s\times H^{s-1}$ with $s>7/2$.
 It is the first global low regularity existence result for nonlinear elastic waves.
A almost global low regularity existence result has been got by the authors \cite{almost}.
\par
The outline of this paper is as follows. The remainder of this introduction will be devoted to the description of
notation which will be used in the sequel, some reduction of the equations of motion and a statement of the low regularity global existence theorem. In Section 2, some necessary tools used to prove the global existence theorem are introduced, including some properties of the null condition, weighted Sobolev inequalities, the bounds on the Klainerman-Sideris type energy and KSS type energy. Finally, the proof of
global low regularity existence theorem will be given in Section 3.

\subsection{Notation}
Denote the space gradient and space-time gradient by $\nabla=(\partial_1,\partial_2,\partial_3)$ and
$
\partial =(\partial_t,\nabla),
$
respectively.
The scaling operator is the vector field
\begin{align}\label{scalinfghj}
S=t\partial_t+r\partial_r,
\end{align}
where $r=|x|,~ \partial_r=\omega\cdot \nabla, ~\omega=(\omega_1,\omega_2,\omega_3), ~\omega_i=x_i/r$.
Denote the collection of vector fields~$Z=(Z_1, Z_2, Z_3, Z_4)=(\nabla,{S})$.
  ~For any multi-index $a=(a_1, a_2, a_3, a_4)$, we denote an ordered product of vector fields $Z^{a}=Z_1^{a_1}Z_2^{a_2}Z_3^{a_3} Z_4^{a_4}=\partial_1^{a_1}\partial_2^{a_2}\partial_3^{a_3} {S}^{a_4}$. Moreover $b\leq a$ means~ $b_{i}\leq a_i$ for each $i=1,\dots,4$.\par
The energy associated to the linear wave operator is
\begin{align}
E_{1}(u(t))=\frac{1}{2}\int_{\mathbb{R}^{3}} \big(|\partial_tu(t,x)|^{2}+ |\nabla u(t,x)|^{2}\big)\, dx,
\end{align}
and the corresponding $k$-th order energy is given by
\begin{align}\label{kord}
E_{k}(u(t))=\sum_{\substack{|a|\leq k-1\\a_4\leq 1}} {E}_1(Z^{a}u(t)).
\end{align}
\par
 We will use the so called \lq\lq good derivatives"
 \begin{align}\label{goodder}
 T=(T_1,T_2,T_3)=\omega\partial_t+\nabla.
 \end{align}
 Let $\sigma=t-r$, $q(\sigma)=\arctan\sigma,
 q'(\sigma)=\frac{1}{1+\sigma^2}=\langle t-r\rangle^{-2}$, where~$\left<\cdot\right>=(1+|\cdot|^2)^{1/2}$.  Since $q$ is bounded, there exists a constant $c>1$, such that
\begin{align}\label{noting}
c^{-1}\leq e^{-q(\sigma)}\leq c.
\end{align}
 Define the ghost weight energy (see \cite{Alinhac01, MR2666888})
 \begin{align}
\mathcal {E}_1(u(t))=\frac{1}{2}\int_{\mathbb{R}^3}  e^{-q(\sigma)}\left<t-r\right>^{-2}{|Tu|^2}\, dx,
 \end{align}
 and its higher-order version
 \begin{align}\label{hkksss}
\mathcal {E}_k(u(t))= \sum_{\substack{|a|\leq k-1\\a_4\leq 1}}\mathcal {E}_1(Z^{a}u(t)).
 \end{align}
\par

The KSS type energy (see \cite{MR2919103,MR1945285,MR2015331,MR2217314}) is defined through
\begin{align}
      \mathcal {M}_{k}(u(t))=\int_0^t  \mathcal {N}_{k}(u(\tau))d\tau,
      \end{align}
      where
\begin{align}\label{morasdd}
  \mathcal {N}_{k}(u(t))=\sum_{\substack{|a|\leq k-1\\a_4\leq 1}}\big\|\left<r\right>^{-1/4}r^{-1/4}\partial Z^{a}u\big\|^2_{L^2(\mathbb{R}^3)}+\sum_{\substack{|a|\leq k-1\\a_4\leq 1}}\big\|\left<r\right>^{-1/4}r^{-5/4} Z^{a}u\big\|^2_{L^2(\mathbb{R}^3)},
\end{align}
and its local version is
      \begin{align}\label{loacljk}
       \mathcal {L}_k(u(t))=\sum_{\substack{|a|\leq k-1\\a_4\leq 1}}\big\|r^{-1/4}\partial Z^{a}u\big\|^2_{L^2(0,t;L^2(|x|\leq 1))}+\sum_{\substack{|a|\leq k-1\\a_4\leq 1}}\big\|r^{-5/4} Z^{a}u\big\|^2_{L^2(0,t;L^2(|x|\leq 1))}.
\end{align}
\par
We will also use the following Klainerman-Sideris type energy (see \cite{Klainerman96})
\begin{align}
{\mathcal {X}}_{k}(u(t))=\sum_{|a|\leq k-2}\big\|\left<t-r\right>\partial\nabla \nabla^{a}u(t)\big\|_{L^2(\mathbb{R}^{3})}.
\end{align}
\par
Consider the space $X^{k}(T)$, which is obtained
by closing the set $C^{\infty}([0,T);C_{0}^{\infty}(\mathbb{R}^3;\mathbb{R}^3))$
 in the norm $\sup\limits_{0\leq t< T}E^{1/2}_{k}(u(t))$.

The solution will be constructed in the space
\begin{align}
X_{\text{rad}}^{k}(T)=\big\{ u: u ~\text{is radially symmetric}, u\in X^{k}(T)\big\}.
\end{align}
\par
Though we will consider radially symmetric solutions, the generators of spatial rotations and
simultaneous rotations will be also important in our analysis.
The angular momentum operators (generators of the spatial rotations) are the vector fields
\begin{align}\label{rotations111}
\Omega=(\Omega_1,\Omega_2,\Omega_3)=x\wedge \nabla,
\end{align}
$\wedge$ being the usual vector cross product. The spatial derivatives can be conveniently decomposed into radial and angular components
\begin{align}\label{angularradial}
\nabla =\omega\partial_r-\frac{\omega}{r}\wedge \Omega.
\end{align}
 Noting the definitions \eqref{goodder} and \eqref{rotations111},
we also have the relationship between the angular momentum operators $\Omega$ and the good derivatives $T$
\begin{align}
\frac{\Omega}{r}=\omega\wedge T,
\end{align}
which implies
\begin{align}\label{haohaohao}
\big|\frac{\Omega}{r} u\big|\leq |Tu|.
\end{align}
Denote the generators of
simultaneous rotations by $\widetilde{\Omega}=(\widetilde{\Omega}_1,\widetilde{\Omega}_2,\widetilde{\Omega}_3)$, where
\begin{align}\label{111222}
\widetilde{\Omega}_{i}=\Omega_{i} I+U_{i},
\end{align}
with
$U_{1}=e_2\otimes e_3-e_3\otimes e_2, U_{2}=e_3\otimes e_1-e_1\otimes e_3, U_{3}=e_1\otimes e_2-e_2\otimes e_1$, $\{e_i\}_{i=1}^{3}$ is the standard basis
on $\mathbb{R}^3$ and $\otimes$ stands for the tensor product on $\mathbb{R}^3$. It is easy to verify the following commutation relationship
\begin{align}\label{compoi1}
\big[\widetilde{\Omega}, \nabla\big]=\nabla,~\big[\widetilde{\Omega}, {S}\big]=0.
\end{align}

\subsection{The equations of motion}
Now we consider the equations of motion for 3-D homogeneous, isotropic and hyperelastic elastic waves. First we have the Lagrangian
 \begin{align}\label{lagrange}
 \mathscr{L}(u)=\iint \frac{1}{2}|u_t|^2-W(\nabla u)~ dxdt,
\end{align}
where $u(t,x)=(u^{1}(t,x),u^{2}(t,x),u^{3}(t,x))$ denotes the displacement vector from the
reference configuration and $W$ is the stored energy function. Since we will consider small solutions, we can write
\begin{align}
W(\nabla u)=l_2(\nabla u)+l_3(\nabla u)+ h.o.t.,
\end{align}
with $l_2(\nabla u)$ and $l_3(\nabla u)$ standing for the quadratic and cubic term in $\nabla u$, respectively, and $h.o.t.$ denoting higher order terms.
Using the frame indifference and isotropic assumption, we can get
\begin{align}\label{L2}
 l_2(\nabla u)=\frac{c_2^2}{2}|\nabla  u|^2+\frac{c_1^2-c_2^2}{2}(\nabla \cdot u)^2,
 \end{align}
 where the material constants $c_1$ (pressure wave speed) and $c_2$ (shear wave speed) satisfy $0<c_2<c_1$,
 and\footnote{We use the summation convention over repeated indices.}
\begin{align}\label{sanjie}
 l_3(\nabla u)&=d_1(\nabla \cdot u)^3+d_2(\nabla \cdot u)|\nabla\wedge u|^2+d_3(\nabla \cdot u)Q_{ij}(u^i,u^j)\nonumber\\
 &~+d_4(\partial_{k}u^{j})Q_{ij}(u^{i},u^{k})+d_5(\partial_{k}u^{j})Q_{ik}(u^{i},u^{j}),
 \end{align}
where $d_i~(i=1,\dots,5)$ are also some constants which only depend on the stored energy function, and the null form $Q_{ij}(f,g)=\partial_i f\partial_jg-\partial_jf\partial_ig$.
\par
By the Hamilton's principle we get the nonlinear elastic wave equation in 3-D
\begin{align}\label{EQwave}
 \partial_t^2u-c_2^2\Delta u-(c_1^2-c_2^2)\nabla\nabla\cdot u=F(\nabla u, \nabla^2 u),
 \end{align}
 where\footnote{In 3-D case, the global existence of small
solutions to quasilinear hyperbolic systems hinges on the specific form
of the quadratic part of the nonlinearity in relation to the linear part. From the analytical point of view, therefore, it is enough to truncate at cubic order in $u$, the higher order corrections having no influence on the existence of small
solutions.}
 \begin{align}\label{nqsddd}
 F(\nabla u, \nabla^2 u)=3d_1\nabla (\nabla\cdot u)^2+d_2\big(\nabla|\nabla\wedge u|^2-2\nabla\wedge(\nabla\cdot u \nabla\wedge u)\big)+Q(u, \nabla u),
 \end{align}
with
 \begin{align}
Q(u, \nabla u)^{i}&=(2d_3+d_4)\big(Q_{ij}(\partial_ku^{k},u^{j})-Q_{jk}(\partial_iu^{k},u^{j})\big)\nonumber\\
&+d_5\big(Q_{ij}(\partial_ju^{k},u^{k})+2Q_{jk}(\partial_ju^{i},u^{k}) -Q_{jk}(\partial_ju^{k},u^{i})\big).
 \end{align}
 The derivation of the exact form of $F(\nabla u, \nabla^2 u)$ can be also found in Agemi \cite{Agemi00}. Following Agemi \cite{Agemi00}, we say that the nonlinear elastic wave equation \eqref{EQwave} satisfies the null condition if
 \begin{align}\label{null1111234}
 d_1=0.
 \end{align}
\par
Consider the Cauchy problem of \eqref{EQwave}--\eqref{nqsddd} with initial data
\begin{align}\label{cauchydata1}
t=0: u=u_0,~u_t=u_1.
\end{align}
\par
Noting the nonlinear elastic wave equation is invariant under simultaneous rotations (see page 860 of \cite{Sideris00}), we can seek radially symmetric solution for the Cauchy problem \eqref{nqsddd}--\eqref{cauchydata1} with radially symmetric initial data.\footnote{A vector function $v$ is called radial if it has the form $v(x)=x\phi(r)~ (r=|x|)$, where $\phi$ is a scalar radial function. We refer the reader to Definition 4.4 and Lemma 4.5 of \cite{MR1774100}.} In the radially symmetric
case, we first derive some equivalent forms of the equation of motion and the null condition.  \par
 If $u\in X^3_{\rm rad}(T)$ is a radially symmetric solution to \eqref{EQwave}, then there exists a scalar function $\psi$ such that
 \begin{align}\label{biao}
 u(t,x)=x\psi(t,r).
 \end{align}
By~\eqref{biao}, it is easy to see that
\begin{align}\label{curl}
\nabla\wedge u =0.
\end{align}
It follows from the Hodge decomposition and \eqref{curl} that
\begin{align}\label{div}
\Delta u=\nabla \nabla\cdot u-\nabla\wedge\nabla\wedge  u=\nabla \nabla\cdot u.
\end{align}
Thanks to \eqref{div}, for the linear part of \eqref{EQwave}, we have
\begin{align}\label{curl11}
\partial^2_t u-c_2^2 \Delta  u-(c_1^2-c_2^2)\nabla \nabla\cdot u=\partial^2_t u-c_1^2 \Delta  u.
\end{align}
Next we consider the nonlinear part $F(\nabla u,\nabla^2 u)$. Obviously \eqref{curl} gives that the second term on the right-hand side of  \eqref{nqsddd} vanishes. Furthermore, the null condition \eqref{null1111234} implies that the first term on the right-hand side of \eqref{nqsddd} also vanishes.
Without loss of generality we can take $c_1=1$. Hence $u$ satisfies
\begin{align}\label{Cauchynew123}
\partial_t^2 u-\Delta u=Q(u, \nabla u).
 \end{align}
\par
Conversely, if $u\in X^3_{\rm rad}(T)$ is a radially symmetric solution to the equation \eqref{Cauchynew123}, then $u$ is also
a radially symmetric solution to the equation \eqref{EQwave} with the null condition \eqref{null1111234}.\par
For the convenience of the following analysis, we rewrite \eqref{Cauchynew123} as
\begin{align}\label{Cauchynew}
\partial_t^2 u-\Delta u= {N}(u,u),
 \end{align}
 where
\begin{align}\label{cutifof99ggneww334}
 N( u, v)^{i}=\partial_{l}({g}^{ijk}_{lmn}\partial_m u^{j}\partial_nv^{k}),
 \end{align}
 and $g=({g}^{ijk}_{lmn})$ is determined via
 $
 {N}(u,u)=Q(u, \nabla u).
$
We can verify that the coefficients are symmetric with respect to pairs of indices
 \begin{align}\label{sym1new}
 g^{ijk}_{lmn}=g^{jik}_{mln}=g^{kji}_{nml},
 \end{align}
 which implies that
 \begin{align}\label{symmert}
 N( u, v)=N( v, u),
 \end{align}
 and the following null condition holds (see Lemma 3.1 of \cite{Agemi00})
 \begin{align} \label{null1100new}
g^{ijk}_{lmn}\omega_l\omega_m\omega_n=0,~\forall~1\leq i,j,k\leq 3, \omega\in S^2.
\end{align}
We note that \eqref{null1100new} just coincides with the definition of standard null condition in the pioneering works Klainerman \cite{Klainerman86} and Christodoulou \cite{Christodoulou86}
 for systems of quasilinear wave equations with single wave speed.
\par
The above argument gives that in order to seek radially symmetric solutions to the Cauchy problem of nonlinear elastic waves \eqref{EQwave}--\eqref{cauchydata1} with the null condition \eqref{null1111234},
we only need to study the radially symmetric solutions to the Cauchy problem \eqref{Cauchynew}--\eqref{cauchydata1} with the null condition
 \eqref{null1100new}.

\subsection{Global low regularity existence theorem}
 The main result in this paper is the following
 \begin{thm}\label{mainthm33}
 Consider the Cauchy problem for {\rm {3}}-D nonlinear elastic waves \eqref{EQwave}--\eqref{cauchydata1}. Assume that the null condition \eqref{null1111234} holds.
Then there exists a constant $\varepsilon_0>0$ such that for any given $\varepsilon$ with $0<\varepsilon\leq \varepsilon_{0}$, if the initial data is radially symmetric and satisfies
\begin{align}\label{label345newe}
\sum_{|a|\leq 2}\|\langle x\rangle\nabla^{a}\nabla u_0\|_{L^2(\mathbb{R}^3)}+\sum_{|a|\leq 2}\|\langle x\rangle\nabla^{a}u_1\|_{L^2(\mathbb{R}^3)}\leq \varepsilon,
\end{align}
then Cauchy problem \eqref{nqsddd}--\eqref{cauchydata1} admits a unique global solution $u\in X_{\rm rad}^{3}(T)$ for every $T>0$.
\end{thm}
\begin{rem}
 In the smallness condition \eqref{label345newe} of the initial data, we only need a low weight $\langle x\rangle$. This is achieved by using the scaling operator \eqref{scalinfghj} only one time in our argument {\rm{(}}note that in the definitions of higher-order energies \eqref{kord}, \eqref{hkksss}, \eqref{morasdd} and \eqref{loacljk}, we restrict $a_4\leq 1${\rm{)}}.
\end{rem}
\section{Preliminaries}
In this section, first we will give some properties on the null condition and some weighted Sobolev inequalities, which will play key roles in the decay estimate under rather low regularity. Then we will give some a priori estimates on the Klainerman-Sideris type energy $\mathcal {X}_{3}(u(t))$ and the KSS type energy $\mathcal {M}_{3}(u(t))$. It turns out that they can all be controlled by the general energy $E_{3}(u(t))$, if it is sufficiently small.
\subsection{Null condition}

Denote the trilinear form
\begin{align}
\widetilde{N}(u,v,w)=g^{ijk}_{lmn}\partial_lu^{i}\partial_mv^{j}\partial_nw^{k}.
\end{align}

\begin{proposition}
If the null condition \eqref{null1100new} is satisfied, then for any $u,v,w\in C^{\infty}(\mathbb{R}^{3};\mathbb{R}^{3})$, we have
\begin{align}\label{xuyao1}
|\langle u,{N}(v,w)\rangle|&\leq \frac{C}{r}|u|\big(|\nabla \Omega v||\nabla w|+|\nabla \Omega w||\nabla v|+|\Omega v||\nabla^2 w|+|\Omega w||\nabla^2 v|\big),\\\label{xuyao2}
|\widetilde{N}(u,v,w)|&\leq \frac{C}{r}\big(|\Omega u||\nabla v||\nabla w|+|\nabla u||\Omega v||\nabla w|+|\nabla u||\nabla v||\Omega w|\big).
\end{align}
\begin{proof}
It follows from the radial-angular decomposition \eqref{angularradial} that we have the following pointwise equality
 \begin{align}\label{contr90}
\widetilde{N}(u,v,w)&=-g^{ijk}_{lmn}(\frac{\omega}{r}\wedge \Omega)_lu^{i}\partial_mv^{j}\partial_nw^{k}-g^{ijk}_{lmn}\omega_l\partial_ru^{i}(\frac{\omega}{r}\wedge \Omega)_mv^{j}\partial_nw^{k}\nonumber\\
&-g^{ijk}_{lmn}\omega_l\partial_ru^{i}\omega_m\partial_rv^{j}(\frac{\omega}{r}\wedge \Omega)_nw^{k}+g^{ijk}_{lmn}\omega_l\omega_m\omega_n\partial_ru^{i}\partial_rv^{j}\partial_rw^{k}.
 \end{align}
Thus \eqref{xuyao2} follows from \eqref{contr90} and the null condition \eqref{null1100new}.
\eqref{xuyao1} can be proved similarly.
\end{proof}
\end{proposition}

\begin{Lemma}\label{Nulll}
Assume that the null condition \eqref{null1100new} is satisfied and $u\in {X}^{3}(T)$~is a solution to \eqref{Cauchynew}.
Then for any multi-index $a=(a_1,a_2,a_3,a_4), a_4\leq 1$,
 we have
\begin{align}
\Box Z^{a}u=\sum_{b+c+d=a}N_{d}(Z^{b}u,Z^{c}u),
\end{align}
where each $N_{d}$ is a quadratic nonlinearity of the form \eqref{cutifof99ggneww334} satisfying the null condition \eqref{null1100new}
. Moreover, if $b+c=a$, then $N_{d}=N$.
\end{Lemma}
\begin{proof}
See Lemma 6.6.5 of H\"{o}rmander \cite{MR1466700} and Lemma 4.1 of Sideris and Tu \cite{Sideris01}.
\end{proof}
\subsection{Weighted Sobolev inequalities}
\begin{Lemma}
For~$u\in C_{0}^{\infty}(\mathbb{R}^{3};\mathbb{R}^{3})$, we have
\begin{align}\label{weight1}
\|ru\|_{L^{\infty}}&\leq C\sum_{|a|\leq 1} \|\nabla \widetilde{\Omega}^au\|_{L^2}+C\sum_{|a|\leq 2} \| \widetilde{\Omega}^au\|_{L^2},  \\\label{weight2}
\|r^{1/2}u\|_{L^{\infty}}&\leq C\sum_{|a|\leq 1} \|\nabla \widetilde{\Omega}^au\|_{L^2},\\\label{charu}
\|r^{1/4}u\|_{L^{\infty}(r\leq 1)}&\leq C\sum_{|a|\leq 1}\|\langle r\rangle^{-1/4}r^{-1/4}\nabla\widetilde{\Omega}^{a}u\|_{L^2}+C\sum_{|a|\leq 1}\|\langle r\rangle^{-1/4}r^{-1/4}\widetilde{\Omega}^{a}u\|_{L^2}\nonumber\\
&+C\sum_{|a|\leq 1}\|\langle r\rangle^{-1/4}r^{-5/4}\widetilde{\Omega}^{a}u\|_{L^2},\\\label{weight3}
\|r\langle t-r \rangle u\|_{L^{\infty}}&\leq C\sum_{|a|\leq 1} \|\langle t-r\rangle\nabla \widetilde{\Omega}^au\|_{L^2}+C\sum_{|a|\leq 2} \|\langle t-r\rangle \widetilde{\Omega}^au\|_{L^2},\\\label{weight4}
\|r^{1/2}\langle t-r \rangle u\|_{L^{\infty}}&\leq C\sum_{|a|\leq 1}  \| \widetilde{\Omega}^au\|_{L^2}+C\sum_{|a|\leq 1} \|\langle t-r\rangle \nabla\widetilde{\Omega}^au\|_{L^2}, \\\label{weight5}
\|\langle t-r \rangle u\|_{L^{\infty}}&\leq C\|\nabla u\|_{L^2}+ C\|\langle t-r\rangle\nabla^2 u\|_{L^2}.
\end{align}
\end{Lemma}
\begin{proof}
For \eqref{weight1}, \eqref{weight2} and \eqref{weight3}, see Lemma 3.3 of Sideris \cite{Sideris00}. For \eqref{charu}, noting that
\begin{align}
\|r^{1/4}u\|_{L^{\infty}(r\leq 1)}\leq C\|r^{1/2}\langle r\rangle^{-1/4}r^{-1/4}u\|_{L^{\infty}(\mathbb{R}^{3})},
\end{align}
we can prove \eqref{charu} by replacing $u$ by $\langle r\rangle^{-1/4} r^{-1/4} u$ in \eqref{weight2}.
 \eqref{weight4} can be found in Lemma 4.1 of Hidano \cite{MR2053322}.  As for \eqref{weight5}, see (2.13) of Hidano \cite{Hidano} and (37) of Zha \cite{MR3549405}.
\end{proof}
\begin{proposition}
For radially symmetric~$u\in C_{0}^{\infty}(\mathbb{R}^{3};\mathbb{R}^{3})$, we have
\begin{align}\label{weight11}
\|r\partial Z^au\|_{L^{\infty}}&\leq CE_{3}^{1/2}(u(t)),~|a|\leq 1,  \\\label{weight22}
\|r^{1/2}Z^au\|_{L^{\infty}}&\leq  CE_{3}^{1/2}(u(t)),~|a|\leq 2,a_4\leq 1,\\\label{charu00}
\|r^{1/4}\nabla Z^{a}u\|_{L^{\infty}(r\leq 1)}&\leq C\sum_{|b|\leq 1}\|\langle r\rangle^{-1/4}r^{-1/4}\nabla^2 Z^{b}u\|_{L^2}+C\sum_{|b|\leq 1}\|\langle r\rangle^{-1/4}r^{-1/4}\nabla Z^{b}u\|_{L^2}\nonumber\\
&+C\sum_{|b|\leq 1}\|\langle r\rangle^{-1/4}r^{-5/4}\nabla Z^{b}u\|_{L^2},~|a|\leq 1, \\\label{weight33}
\|r\langle t-r \rangle\partial \nabla u\|_{L^{\infty}}&\leq C\mathcal {X}_{3}(u(t)),\\\label{weight44}
\|r^{1/2}\langle t-r \rangle \partial Z^au\|_{L^{\infty}}&\leq CE_{2}^{1/2}(u(t))+C\mathcal {X}_{3}(u(t)),~|a|\leq 1,  a_4=0, \\\label{weight55}
\|\langle t-r \rangle\nabla u\|_{L^{\infty}}&\leq CE_{2}^{1/2}(u(t))+C\mathcal {X}_{3}(u(t)).
\end{align}
\end{proposition}
\begin{proof}
\eqref{weight55} is a direct application of \eqref{weight5}.  Since $u$ is radially symmetric, $\widetilde{\Omega}u=0$. Noting the commutation relationship \eqref{compoi1}, we can derive \eqref{weight11}--\eqref{weight44} from \eqref{weight1}--\eqref{weight4}, respectively.
\end{proof}

\subsection{Klainerman-Sideris bound}
\begin{Lemma}\label{ks}
Let radially symmetric $u\in C^{\infty}([0,T);C_{0}^{\infty}(\mathbb{R}^3;\mathbb{R}^3))$. Then
\begin{align}\label{non}
{\mathcal {X}}_3(u(t))\leq C E^{1/2}_{3}(u(t))+Ct\sum_{|a|\leq 1}\|\Box \nabla^{a}u(t)\|_{L^{2}(\mathbb{R}^3)}.
\end{align}
\end{Lemma}
\begin{proof}
See Lemma 3.1 of Klainerman and Sideris \cite{Klainerman96}.
\end{proof}
\begin{proposition}\label{prop2233}
 Let $u$ be a radially symmetric and smooth solution to equation \eqref{Cauchynew}.  Assume that
\begin{align}\label{small234}
\varepsilon_1\equiv \sup_{0\leq t<T}E^{1/2}_{3}(u(t))
\end{align}
is sufficiently small. Then for $0\leq t< T$, we have
\begin{align}\label{high1}
{\mathcal {X}}_3(u(t))\leq CE_{3}^{1/2}(u(t)).
\end{align}
\end{proposition}
\begin{proof}
By Lemma \ref{ks}, we need to estimate the second term on the right-hand side of \eqref{non}.
For $|a|\leq 1$, we have
\begin{align}\label{keye}
\|\Box \nabla^{a}u(t)\|_{L^{2}}\leq C\sum_{b+c=a}\|N(\nabla^{b}u,\nabla^{c}u)\|_{L^2}.
\end{align}
Noting \eqref{symmert}, we only need to consider the case $b=0, c=a$. It is obvious that
\begin{align}\label{right345}
\|N(u,\nabla^{a}u)\|_{L^2}\leq C\|\nabla u\nabla^2\nabla^{a}u\|_{L^2}+C\|\nabla^2 u\nabla \nabla^{a}u\|_{L^2}.
\end{align}
\par
We first consider on the region $r\geq \langle t \rangle/{2}$. It follows from \eqref{weight11} that
\begin{align}
&\|\nabla u\nabla^2\nabla^{a}u\|_{L^2(r\geq \langle t \rangle/{2})}+\|\nabla^2 u\nabla \nabla^{a}u\|_{L^2(r\geq \langle t \rangle/{2})}\nonumber\\
&\leq C\langle t\rangle^{-1}\|r\nabla u\nabla^2\nabla^{a}u\|_{L^2(r\geq \langle t \rangle/{2})}+C\langle t\rangle^{-1}\|r\nabla^2 u\nabla \nabla^{a}u\|_{L^2(r\geq \langle t \rangle/{2})}\nonumber\\
&\leq C\langle t\rangle^{-1}\|r\nabla u\|_{L^{\infty}}\|\nabla^2\nabla^{a}u\|_{L^2}+C\langle t\rangle^{-1}\|r\nabla^2 u\|_{L^{\infty}}\|\nabla \nabla^{a}u\|_{L^2}\nonumber\\
&\leq C\langle t\rangle^{-1}E_{3}(u(t)).
\end{align}
While on the region $r\leq \langle t \rangle/{2}$, by \eqref{weight55} we have
\begin{align}
&\|\nabla u\nabla^2\nabla^{a}u\|_{L^2(r\leq \langle t \rangle/{2})}\nonumber\\
&\leq C\langle t\rangle^{-1}\|\langle t-r\rangle\nabla u\nabla^2\nabla^{a}u\|_{L^2(r\leq \langle t \rangle/{2})}\nonumber\\
&\leq C\langle t\rangle^{-1}\|\langle t-r\rangle\nabla u\|_{L^{\infty}}\|\nabla^2\nabla^{a}u\|_{L^2}\nonumber\\
&\leq C\langle t\rangle^{-1}\big(E^{1/2}_{2}(u(t))+{\mathcal {X}}_{3}(u(t))\big)E^{1/2}_{3}(u(t)),
\end{align}
and it follows from the Sobolev embedding $H^1(\mathbb{R}^3)\hookrightarrow L^6(\mathbb{R}^3)$ and $H^1(\mathbb{R}^3)\hookrightarrow L^3(\mathbb{R}^3)$ that
\begin{align}\label{finalerr}
&\|\nabla^2 u\nabla \nabla^{a}u\|_{L^2(r\leq \langle t \rangle/{2})}\nonumber\\
&\leq C\langle t\rangle^{-1}\|\langle t-r\rangle\nabla^2 u\nabla \nabla^{a}u\|_{L^2(r\leq \langle t \rangle/{2})}\nonumber\\
&\leq C\langle t\rangle^{-1}\|\langle t-r\rangle\nabla^2 u\|_{L^{6}}\|\nabla \nabla^{a}u\|_{L^3}\nonumber\\
&\leq C\langle t\rangle^{-1}\|\langle t-r\rangle\nabla^2 u\|_{{H}^{1}}\|\nabla \nabla^{a}u\|_{H^1}\nonumber\\
&\leq C\langle t\rangle^{-1}\big(E^{1/2}_{2}(u(t))+{\mathcal {X}}_{3}(u(t))\big)E^{1/2}_{3}(u(t)).
\end{align}
\par
Hence it follows from \eqref{non}, \eqref{small234}, \eqref{keye}--\eqref{finalerr} that
\begin{align}\label{high13455t}
{\mathcal {X}}_3(u(t))&\leq CE_{3}^{1/2}(u(t))+CE_{3}(u(t)) +CE^{1/2}_{3}(u(t)){\mathcal {X}}_{3}(u(t))\nonumber\\
&\leq CE_{3}^{1/2}(u(t))+C\varepsilon_{1}E_{3}^{1/2}(u(t)) +C\varepsilon_{1}{\mathcal {X}}_{3}(u(t)).
\end{align}
If $\varepsilon_1$ is sufficiently small, we can get \eqref{high1}.
\end{proof}
\par
\subsection{KSS bound}

\begin{Lemma}\label{thmKSS00}
Consider the following perturbed linear wave operator
\begin{align}\label{elasticwave345}
\Box_{h}u=\Box u+Hu,
\end{align}
where the perturbed term
\begin{align}\label{Hdefin}
(Hu)^{i}&=\partial_{l}(h^{ij}_{lm}(t,x)\partial_mu^{j}),~~i=1,2,3.
\end{align}
Assume that~$h=(h^{ij}_{lm})$ satisfies the following symmetric condition
\begin{align}\label{duichengxing678}
h^{ij}_{lm}=h^{ji}_{ml}
\end{align}
and the smallness condition
\begin{align}\label{smallness}
|h|=\sum_{i,j,l,m=1}^{3}|h^{ij}_{lm}|\ll 1.
\end{align}
 Suppose that $u\in C^{\infty}([0,T);C_{0}^{\infty}(\mathbb{R}^3;\mathbb{R}^3))$.  Then for any $0<t<T$ we have
 \begin{align}
&\big(\log(2+t)\big)^{-1}\|\left<r\right>^{-1/4}r^{-1/4}\partial u\|^2_{L^2_{t,x}(S_{t})}+\big(\log(2+t)\big)^{-1}\|\left<r\right>^{-1/4}r^{-5/4} u\|^2_{L^2_{t,x}(S_{t})}\nonumber\\
&+\|r^{-1/4}\partial u\|^2_{L^2(0,t;L^2(|x|\leq 1))}+\|r^{-5/4} u\|^2_{L^2(0,t;L^2(|x|\leq 1))}\nonumber\\
&\leq C\|\partial u(0)\|_{L^2}^2+C\big\|\big(|\nabla u|+\langle r\rangle^{-1/2}r^{-1/2}|u|\big)\Box_h u\big\|_{L^1_{t,x}(S_{t})}\nonumber\\
&+C\big\|\big(|\partial h|+\langle r\rangle^{-1/2}r^{-1/2}|h|\big)|\nabla u|\big(|\nabla u|+r^{-1}| u|\big)\big\|_{L^1_{t,x}(S_{t})},
 \end{align}
 where the strip $S_t=[0,t]\times \mathbb{R}^3$.
\end{Lemma}
\begin{proof}
See \ref{appen}.
\end{proof}
\begin{proposition}\label{prop2233kss}
 Let $u$ be a radially symmetric and smooth solution to equation \eqref{Cauchynew}.  Assume that
\begin{align}\label{small23456666}
\varepsilon_1\equiv \sup_{0\leq t<T}E^{1/2}_{3}(u(t))
\end{align}
is sufficiently small. Then for $0\leq t< T$, we have
   \begin{align}\label{high1666666}
     {\mathcal {M}}^{1/2}_3(u(t))\leq C\big(\log(2+t)\big)^{1/2}(1+t)^{\delta}\sup_{0\leq \tau\leq t}E_{3}^{1/2}(u(\tau)),
    \end{align}
with $\delta=\varepsilon_1/2$.
\end{proposition}
\begin{proof}
By Lemma \ref{Nulll} and Lemma \ref{thmKSS00}, we have that
\begin{align}\label{ghjutn}
&\big(\log(2+t)\big)^{-1}{\mathcal {M}}_3(u(t))+{\mathcal {L}}_3(u(t))\nonumber\\
&\leq CE_{3}(u(0))+C\sum_{\substack{|a|\leq 2\\a_4\leq 1}}\big\|\big(|\partial\nabla u|+\langle r\rangle^{-1/2}r^{-1/2}|\nabla u|\big)|\nabla Z^au|\big(|\nabla Z^au|+r^{-1}| Z^au|\big)\big\|_{L^1_{t,x}(S_{t})}\nonumber\\
&+C\sum_{\substack{|a|\leq 2\\a_4\leq 1}}\sum_{\substack{b+c+d=a\\b,c\neq a}}\big\|\big(|\nabla Z^au|+\langle r\rangle^{-1/2}r^{-1/2}|Z^au|\big)N_d(Z^bu,Z^cu)\big\|_{L^1_{t,x}(S_{t})}.
\end{align}
\par
We first estimate all the terms on the right-hand side of \eqref{ghjutn} on the region $r\geq 1/2$. Note that when $r\geq 1/2$, $\langle t\rangle\leq Cr\langle t-r\rangle$. For any $|a|\leq 2, a_4\leq 1$, it follows from \eqref{weight33}, \eqref{weight55}, the Hardy inequality and Proposition \ref{prop2233} that
\begin{align}
&\big\|\big(|\partial\nabla u|+\langle r\rangle^{-1/2}r^{-1/2}|\nabla u|\big)|\nabla Z^au|\big(|\nabla Z^au|+r^{-1}| Z^au|\big)\big\|_{L^1_{x}(r\geq 1/2)}\nonumber\\
&\leq \big\|\big(|\partial\nabla u|+r^{-1}|\nabla u|\big)|\nabla Z^au|\big(|\nabla Z^au|+r^{-1}|\nabla Z^au|\big)\big\|_{L^1_{x}(r\geq 1/2)}\nonumber\\
&\leq C\langle t\rangle^{-1}\big\|\big(r\langle t-r\rangle|\partial\nabla u|+\langle t-r\rangle{|\nabla u|}\big)|\nabla Z^au|\big(|\nabla Z^au|+r^{-1}|\nabla Z^au|\big)\big\|_{L^1_{x}(r\geq 1/2)}\nonumber\\
&\leq C\langle t\rangle^{-1}\big(\|r\langle t-r\rangle\partial\nabla u\|_{L^{\infty}}+\|\langle t-r\rangle\nabla u\|_{L^{\infty}}\big)\|\nabla Z^au\|_{L^2}(\|\nabla Z^au\|_{L^2}+\|r^{-1} Z^au\|_{L^2})\nonumber\\
&\leq C\langle t\rangle^{-1}\big({\mathcal {X}}_3(u(t))+E^{1/2}_{2}(u(t))\big)E_{3}(u(t))\nonumber\\
&\leq C\langle t\rangle^{-1}E^{3/2}_{3}(u(t)).
\end{align}
For any $|a|\leq 2, a_4\leq 1, b+c+d=a, b,c\neq a$, if $c_4=0$,
 by \eqref{weight11}, the Hardy inequality and Proposition \ref{prop2233}, we have that it holds that
\begin{align}
&\big\|\big(|\nabla Z^au|+\langle r\rangle^{-1/2}r^{-1/2}|Z^au|\big)N_d(Z^bu,Z^cu)\big\|_{L_{x}^1(r\geq 1)}\nonumber\\
&\leq C\langle t\rangle^{-1}\big\|\big(|\nabla Z^au|+r^{-1}{|Z^au|}\big)r\nabla Z^bu \langle t-r\rangle\nabla^2Z^cu\big\|_{L_{x}^1(r\geq 1)}\nonumber\\
&\leq C\langle t\rangle^{-1}\big(\|\nabla Z^au\|_{L^2}+\|r^{-1}{Z^au}\|_{L^2}\big)\|r\nabla Z^bu\|_{L^{\infty}} \|\langle t-r\rangle\nabla^2Z^cu\|_{L^{2}}\nonumber\\
&\leq C\langle t\rangle^{-1}{\mathcal {X}}_3(u(t))E_{3}(u(t))\nonumber\\
&\leq C\langle t\rangle^{-1}E^{3/2}_{3}(u(t)).
\end{align}
If $b_4=0$, we further need to split the region $r\geq 1/2$ into $1/2\leq r\leq \langle t\rangle/2$ and $r\geq \langle t\rangle/2$. On the region
$1/2\leq r\leq \langle t\rangle/2$,
it follows from \eqref{weight44}, the Hardy inequality and Proposition \ref{prop2233} that
\begin{align}
&\big\|\big(|\nabla Z^au|+\langle r\rangle^{-1/2}r^{-1/2}|Z^au|\big)N_d(Z^bu,Z^cu)\big\|_{L_{x}^1(1/2\leq r\leq \langle t\rangle/2)}\nonumber\\
&\leq C\langle t\rangle^{-1}\big\|\big(|\nabla Z^au|+r^{-1}{|Z^au|}\big)r^{1/2}\langle t-r\rangle\nabla Z^bu \nabla^2Z^cu\big\|_{L_{x}^1(1/2\leq r\leq \langle t\rangle/2)}\nonumber\\
&\leq C\langle t\rangle^{-1}\big(\|\nabla Z^au\|_{L^2}+\|r^{-1}{Z^au}\|_{L^2}\big)\|r^{1/2}\langle t-r\rangle\nabla Z^bu\|_{L^{\infty}} \|\nabla^2Z^cu\|_{L^{2}}\nonumber\\
&\leq C\langle t\rangle^{-1}\big(E^{1/2}_{2}(u(t))+{\mathcal {X}}_3(u(t))\big)E_{3}(u(t))\nonumber\\
&\leq C\langle t\rangle^{-1}E^{3/2}_{3}(u(t)).
\end{align}
And on the region
$ r\geq \langle t\rangle/2$,
it follows from \eqref{weight11}, the Hardy inequality and Proposition \ref{prop2233} that
\begin{align}
&\big\|\big(|\nabla Z^au|+\langle r\rangle^{-1/2}r^{-1/2}|Z^au|\big)N_d(Z^bu,Z^cu)\big\|_{L_{x}^1(r\geq \langle t\rangle/2)}\nonumber\\
&\leq C\langle t\rangle^{-1}\big\|\big(|\nabla Z^au|+r^{-1}{|Z^au|}\big)r\nabla Z^bu \nabla^2Z^cu\big\|_{L_{x}^1(r\geq \langle t\rangle/2)}\nonumber\\
&\leq C\langle t\rangle^{-1}\big(\|\nabla Z^au\|_{L^2}+\|r^{-1}{Z^au}\|_{L^2}\big)\|r\nabla Z^bu\|_{L^{\infty}} \|\nabla^2Z^cu\|_{L^{2}}\nonumber\\
&\leq C\langle t\rangle^{-1}E^{3/2}_{3}(u(t)).
\end{align}
\par
Now we will estimate the terms on the right-hand side of \eqref{ghjutn} on the region $r\leq 1/2$. \eqref{weight22} gives that for any $|a|\leq 2, a_4\leq 1$,
\begin{align}
&\big\|\big(|\partial\nabla u|+\langle r\rangle^{-1/2}r^{-1/2}|\nabla u|\big)|\nabla Z^au|\big(|\nabla Z^au|+r^{-1}| Z^au|\big)\big\|_{L^1_{t,x}(r\leq 1/2)}\nonumber\\
&\leq \big\|\big(|\partial\nabla u|+r^{-1/2}|\nabla u|\big)|\nabla Z^au|\big(|\nabla Z^au|+r^{-1}| Z^au|\big)\big\|_{L^1_{t,x}(r\leq 1/2)}\nonumber\\
&\leq C\big\|\big(r^{1/2}|\partial \nabla u|+|\nabla u|\big)r^{-1/4}|\nabla Z^{a}u|\big(r^{-1/4}|\nabla Z^{a}u|+r^{-5/4}| Z^{a}u|\big)\big \|_{L^1_{t,x}(r\leq 1/2)}\nonumber\\
&\leq C\big(\|r^{1/2}\partial \nabla u\|_{L_{t,x}^{\infty}(S_t)}+\|\nabla u\|_{L_{t,x}^{\infty}(S_t)}\big)\|r^{-1/4}\nabla Z^{a}u\|_{L^2_{t,x}(r\leq 1)}\nonumber\\
&~~\cdot
\big(\|r^{-1/4}\nabla Z^{a}u\|_{L^2_{t,x}(r\leq 1)}+\|r^{-5/4} Z^{a}u\|_{L^2_{t,x}(r\leq 1)}\big)\nonumber\\
&\leq C\sup_{0\leq \tau\leq t}E_{3}^{1/2}(u(\tau))\mathcal {L}_3(u(t)).
\end{align}
Similarly, via \eqref{weight22} we also have that for any $|a|\leq 2, a_4\leq 1, b+c+d=a,b,c\neq a$,
\begin{align}
&\big\|\big(|\nabla Z^au|+\langle r\rangle^{-1/2}r^{-1/2}|Z^au|\big)N_d(Z^bu,Z^cu)\big\|_{L_{t,x}^1(r\leq 1)}\nonumber\\
&\leq C\big\|\big( r^{-1/4}|\nabla Z^au|+r^{-5/4}{|Z^au|}\big) r^{1/2}\nabla Z^bu~ r^{-1/4}\nabla^2Z^cu\big\|_{L_{t,x}^1(r\leq 1)}\nonumber\\
&\leq  C\big( \|r^{-1/4}\nabla Z^au|_{L^2_{t,x}(r\leq 1)}+\|r^{-5/4}{Z^au}\|_{L^2_{t,x}(r\leq 1)}\big) \|r^{1/2}\nabla Z^bu\|_{L^{\infty}_{t,x}(S_t)}\| r^{-1/4}\nabla^2Z^cu\big\|_{L^2_{t,x}(r\leq 1)}\nonumber\\
&\leq C\sup_{0\leq \tau\leq t}E_{3}^{1/2}(u(\tau))\mathcal {L}_3(u(t)).
\end{align}
\par
The above argument gives
\begin{align}
&\big(\log(2+t)\big)^{-1}{\mathcal {M}}_3(u(t))+{\mathcal {L}}_3(u(t))\nonumber\\
&\leq CE_{3}(u(0))+C\int_0^{t}(1+\tau)^{-1}E^{3/2}_{3}(u(\tau))d\tau+C\sup_{0\leq \tau\leq t}E_{3}^{1/2}(u(\tau)){\mathcal {L}}_3(u(t))
\nonumber\\
&\leq CE_{3}(u(0))+C\log(1+t)\sup_{0\leq \tau\leq t}E^{3/2}_{3}(u(\tau))+C\sup_{0\leq \tau\leq t}E_{3}^{1/2}(u(\tau)){\mathcal {L}}_3(u(t))\nonumber\\
&\leq CE_{3}(u(0))+C\varepsilon_1\log(1+t)\sup_{0\leq \tau\leq t}E_{3}(u(\tau))+C\varepsilon_1{\mathcal {L}}_3(u(t)).
\end{align}
If $\varepsilon_1$ is sufficiently small, we have
\begin{align}
&\big(\log(2+t)\big)^{-1}{\mathcal {M}}_3(u(t))+{\mathcal {L}}_3(u(t))\nonumber\\
&\leq CE_{3}(u(0))+C\varepsilon_1\log(1+t)\sup_{0\leq \tau\leq t}E_{3}(u(\tau))\nonumber\\
&\leq CE_{3}(u(0))+C\log(1+t)^{\varepsilon_1}\sup_{0\leq \tau\leq t}E_{3}(u(\tau))\nonumber\\
&\leq CE_{3}(u(0))+C(1+t)^{\varepsilon_1}\sup_{0\leq \tau\leq t}E_{3}(u(\tau))\nonumber\\
&\leq C(1+t)^{\varepsilon_1}\sup_{0\leq \tau\leq t}E_{3}(u(\tau)),
\end{align}
which completes the proof of Proposition \ref{prop2233kss}.
\end{proof}
\section{Proof of Theorem \ref{mainthm33}}
In order to prove Theorem \ref{mainthm33}, the key point is the following a priori estimate.
 \begin{proposition}\label{mainthm33333}
There exist positive constants $\varepsilon_0$ and $A$ such that for any given $\varepsilon$ with $0<\varepsilon\leq \varepsilon_{0}$, if the initial data is radially symmetric and satisfies
\eqref{label345newe}
and $u$ is a smooth and radially symmetric solution
to the Cauchy problem \eqref{EQwave}--\eqref{cauchydata1} with the null condition \eqref{null1111234}, then for any $T>0$,
\begin{align}
\sup_{0\leq t<T}E^{1/2}_{3}(u(t))\leq  A \varepsilon.
\end{align}
\end{proposition}
\par
Based on the method of proving Proposition \ref{mainthm33333},  by some density argument and contraction-mapping argument, we can show Theorem
 \ref{mainthm33}. Because this procedure is routine and in order to keep the paper to a moderate length, we will omit it and refer the reader to \cite{MR2919103} and \cite{MR2262091}.\par
 Now we will prove Proposition \ref{mainthm33333}. We only need to analyze the Cauchy problem \eqref{Cauchynew}--\eqref{cauchydata1} with the null condition \eqref{null1100new}.  Assume that $u=u(t,x)$ is a smooth and radially symmetric solution of the Cauchy problem \eqref{Cauchynew}--\eqref{cauchydata1} on $[0,T)$.
 We will show that there exist positive constants $\varepsilon_0$ and $A$ such that for any $T>0$, we have $\sup\limits_{0\leq t<T}E^{1/2}_{3}(u(t))\leq  A \varepsilon$ under the assumption \eqref{label345newe} and $\sup\limits_{0\leq t<T}E^{1/2}_{3}(u(t))\leq  2A \varepsilon$, where $0<\varepsilon\leq\varepsilon_0$.\par
 \subsection{General energy and ghost weight energy estimates}
 Following Alinhac \cite{Alinhac01,MR2666888}, we will use the ghost weight energy method.
   By Lemma \ref{Nulll}, we have
\begin{align}\label{heng}
 \sum_{\substack{|a|\leq 2\\a_4\leq 1}}\int\langle &e^{-q(\sigma)}\partial_t Z^{a}u, \Box Z^{a}u\rangle \, dx=
\sum_{\substack{|a|\leq 2\\a_4\leq 1}}\sum_{{b+c+d=a}}\int \langle e^{-q(\sigma)}\partial_t Z^{a}u, N_d(Z^{b}u,Z^{c}u)\rangle \, dx.
\end{align}
As we are now treating a quasilinear system, for the right-hand side of \eqref{heng}, special attention should be paid on terms with $b=a$ or $c=a$ with $|a|=2$. Noting \eqref{symmert}, we can rewrite \eqref{heng} as
\begin{align}\label{heng1}
 &\sum_{\substack{|a|\leq 2\\a_4\leq 1}}\int\langle e^{-q(\sigma)}\partial_t Z^{a}u, \Box Z^{a}u\rangle \, dx\nonumber\\
&=2\sum_{\substack{|a|= 2\\a_4\leq 1}}\int \langle e^{-q(\sigma)}\partial_t Z^{a}u, N(u,Z^{a}u)\rangle\, dx+\sum_{\substack{|a|=2\\a_4\leq 1}}\sum_{\substack{b+c+d=a\\b,c\neq a}}\int \langle e^{-q(\sigma)}\partial_t Z^{a}u, N_d(Z^{b}u,Z^{c}u)\rangle \,dx\nonumber\\
&+\sum_{|a|\leq 1}\sum_{b+c+d=a}\int \langle e^{-q(\sigma)}\partial_t Z^{a}u, N_d(Z^{b}u,Z^{c}u)\rangle \,dx.
\end{align}
\par
For the left-hand side of \eqref{heng1}, by integration by parts we have
\begin{align}\label{LEFT}
\int \langle e^{-q(\sigma)}&\partial_t Z^{a}u, \Box Z^{a}u\rangle\, dx\nonumber\\
&=\frac{1}{2}\frac{d}{d t}\int e^{-q(\sigma)}|\partial Z^{a}u|^2\,dx+\frac{1}{2}\int  e^{-q(\sigma)}\left<t-r\right>^{-2}{|TZ^{a}u|^2}\, dx.
\end{align}
\par
For the right-hand side of \eqref{heng1}, it follows from the symmetric conditions \eqref{sym1new} and the integration by parts that
\begin{align}\label{RIGHT}
&\int \langle e^{-q(\sigma)}\partial_t Z^{a}u, N(u,Z^{a}u)\rangle\, dx\nonumber\\
&=g^{ijk}_{lmn}\int e^{-q(\sigma)}\partial_t(Z^{a}u)^i\partial_{l}\big(\partial_mu^{j}\partial_{n}(Z^{a}u)^k\big)\, dx\nonumber\\
&=-g^{ijk}_{lmn}\int e^{-q(\sigma)}\partial_t\partial_{l}(Z^{a}u)^i\partial_mu^{j}\partial_{n}(Z^{a}u)^k\, dx-g^{ijk}_{lmn}\int e^{-q(\sigma)}q'(\sigma)\omega_l\partial_t(Z^{a}u)^i\partial_mu^{j}\partial_{n}(Z^{a}u)^k\, dx\nonumber\\
&=-\frac{1}{2}g^{ijk}_{lmn}\partial_t\int e^{-q(\sigma)}\partial_{l}(Z^{a}u)^i\partial_mu^{j}\partial_{n}(Z^{a}u)^k\, dx
+\frac{1}{2}g^{ijk}_{lmn}\int e^{-q(\sigma)}\partial_l(Z^{a}u)^i\partial_m\partial_tu^{j}\partial_{n}(Z^{a}u)^k\,dx\nonumber\\
&+\frac{1}{2}g^{ijk}_{lmn}\int e^{-q(\sigma)}q'(\sigma)\partial_l(Z^{a}u)^i\partial_mu^{j}\partial_{n}(Z^{a}u)^k\,dx-g^{ijk}_{lmn}\int e^{-q(\sigma)}q'(\sigma)T_l(Z^{a}u)^i\partial_mu^{j}\partial_{n}(Z^{a}u)^k\, dx.
\end{align}
\par
Define the perturbed  energy
\begin{align}\label{pEnrt}
\widetilde{E}_3(u(t))=\frac{1}{2}\sum_{\substack{|a|\leq 2\\ a_4\leq 1}}\int e^{-q(\sigma)}|\partial Z^{a}u|^2\,dx
+\sum_{\substack{|a|=2\\ a_4\leq 1}}g^{ijk}_{lmn}\int e^{-q(\sigma)}\partial_{l}(Z^{a}u)^i\partial_mu^{j}\partial_{n}(Z^{a}u)^k\, dx.
\end{align}
Noting that $\|\partial u(t)\|_{L^{\infty}}\leq CE^{1/2}_{3}(u(t))$, by \eqref{noting}, for small solutions we have
\begin{align}\label{equiee}
(2c)^{-1}{E}_3(u(t))\leq \widetilde{E}_3(u(t))\leq 2c{E}_3(u(t)).
\end{align}
\par
Returning to \eqref{heng1}, by \eqref{LEFT}, \eqref{RIGHT} and \eqref{pEnrt}, we have the following energy identity:
\begin{align}\label{high}
&\frac{d}{dt}\widetilde{E}_3(u(t))+\mathcal {E}_3(u(t))\nonumber\\
&=\sum_{\substack{|a|= 2\\a_4\leq 1}}\int e^{-q(\sigma)}\widetilde{N}(Z^{a}u, \partial_tu, Z^{a}u)\,dx+\sum_{\substack{|a|= 2\\a_4\leq 1}}\int e^{-q(\sigma)}q'(\sigma)\widetilde{N}(Z^{a}u, u, Z^{a}u)\,dx\nonumber\\
&~~~-2\sum_{\substack{|a|= 2\\a_4\leq 1}}g^{ijk}_{lmn}\int e^{-q(\sigma)}q'(\sigma)T_l(Z^{a}u)^i\partial_mu^{j}\partial_{n}(Z^{a}u)^k\, dx\nonumber\\
&~~~
+\sum_{\substack{|a|\leq 2\\a_4\leq 1}}\sum_{\substack{b+c+d=a\\|b|,|c|\leq 1}}\int e^{-q(\sigma)} \langle \partial_t Z^{a}u, N_d(Z^{b}u,Z^{c}u)\rangle \,dx.
\end{align}
\par
Now we are ready to estimate all the terms on the right-hand side of \eqref{high}. For this purpose, we will separate integrals over the regions
$r\geq \langle t\rangle/2$, $1/2\leq r\leq \langle t\rangle/2$ and $r\leq 1/2$, respectively. To get enough decay in time on the region $r\geq \langle t\rangle/2$,  we need to exploit
 the null condition. \par
 {\bf{ On the region $r\geq {\langle t\rangle}/{2}$.}}  Now we consider all the terms on the right-hand side of \eqref{high} on the region $r\geq {\langle t\rangle}/{2}$.
For the first term on the right-hand side of \eqref{high}, for $|a|= 2, a_4\leq 1$, it follows from \eqref{xuyao2} that
\begin{align}
\big|\widetilde{N}(Z^{a}u, \partial_tu, Z^{a}u)\big|\leq \frac{C}{r}\big(|\Omega Z^au||\nabla \partial_tu||\nabla Z^a u|+|\nabla Z^a u||\Omega\partial_t u||\nabla Z^a u|\big).
\end{align}
Noting \eqref{haohaohao}, we have that
\begin{align}\label{lioajie34}
&\big\|e^{-q(\sigma)}\widetilde{N}(Z^{a}u, \partial_tu, Z^{a}u)\big\|_{L^1(r\geq \langle t\rangle/2)}\nonumber\\
&\leq C\| T Z^au \nabla \partial_tu\nabla Z^a u\|_{L^1(r\geq \langle t\rangle/2)}+C\langle t\rangle^{-1}\|\nabla Z^a u\partial_t u\nabla Z^a u\|_{L^1(r\geq \langle t\rangle/2)}.
\end{align}
Via \eqref{weight33} and Proposition \ref{prop2233}, we can get
\begin{align}\label{follw111}
&\| T Z^au \nabla \partial_tu\nabla Z^a u\|_{L^1(r\geq \langle t\rangle/2)}\nonumber\\
&\leq C\langle t\rangle^{-1}\|\langle t-r\rangle^{-1} T Z^au ~r\langle t-r\rangle\nabla \partial_tu\nabla Z^a u\|_{L^1(r\geq \langle t\rangle/2)}
\nonumber\\
&\leq C\langle t\rangle^{-1}\|\langle t-r\rangle^{-1} T Z^au\|_{L^2} \|r\langle t-r\rangle\partial\nabla u\|_{L^{\infty}}\|\nabla Z^a u\|_{L^2}
\nonumber\\
&\leq C\langle t\rangle^{-1}\mathcal {E}_3^{1/2}(u(t))\mathcal {X}_{3}(u(t)) {E}_3^{1/2}(u(t))
\nonumber\\
&\leq C\langle t\rangle^{-1}{E}_3(u(t))\mathcal {E}_3^{1/2}(u(t)).
\end{align}
It follows from \eqref{weight11} that
\begin{align}\label{follw222}
&\langle t\rangle^{-1}\|\nabla Z^a u\partial_t u\nabla Z^a u\|_{L^1(r\geq \langle t\rangle/2)}\nonumber\\
&\leq C\langle t\rangle^{-2}\|\nabla Z^a u~ r\partial_t u\nabla Z^a u\|_{L^1(r\geq \langle t\rangle/2)}\nonumber\\
&\leq C\langle t\rangle^{-2}\|\nabla Z^a u\|_{L^2} \|r\partial_t u\|_{L^{\infty}}\|\nabla Z^a u\|_{L^2}\nonumber\\
&\leq C\langle t\rangle^{-2} {E}_3^{3/2}(u(t)).
\end{align}
So it follows from \eqref{lioajie34}--\eqref{follw222} that
\begin{align}\label{ssss234}
&\sum_{\substack{|a|= 2\\a_4\leq 1}}\big\| e^{-q(\sigma)}\widetilde{N}(Z^{a}u, \partial_tu, Z^{a}u)\big\|_{L^1(r\geq \langle t\rangle/2)}\nonumber\\
&~~~~~~\leq C\langle t\rangle^{-1}{E}_3(u(t))\mathcal {E}_3^{1/2}(u(t))+C\langle t\rangle^{-2} {E}_3^{3/2}(u(t)).
\end{align}
\par
Similarly to the first term, for the second term on the right-hand side of \eqref{high}, for $|a|= 2, a_4\leq 1$, we have
\begin{align}\label{ssss234}
&\|e^{-q(\sigma)}q'(\sigma)\widetilde{N}(Z^{a}u, u, Z^{a}u)\|_{L^1(r\geq \langle t\rangle/2)}\nonumber\\
&\leq C\|q'(\sigma) TZ^{a}u\nabla u\nabla Z^au\|_{L^1(r\geq \langle t\rangle/2)}+C\langle t\rangle^{-1}\|q'(\sigma) \nabla Z^{a}u u\nabla Z^au\|_{L^1(r\geq \langle t\rangle/2)}.
\end{align}
For the first term on the right-hand side of \eqref{ssss234}, by \eqref{weight11} we have
\begin{align}\label{dffggg}
&\|q'(\sigma) TZ^{a}u\nabla u\nabla Z^au\|_{L^1(r\geq \langle t\rangle/2)}\nonumber\\
&\leq C\langle t\rangle^{-1}\|\langle t-r\rangle^{-2} TZ^{a}u~ r\nabla u\nabla Z^au\|_{L^1(r\geq \langle t\rangle/2)}\nonumber\\
&\leq C\langle t\rangle^{-1}\|\langle t-r\rangle^{-1} TZ^{a}u\|_{L^2}\|r\nabla u\|_{L^{\infty}}\|\nabla Z^au\|_{L^2}\nonumber\\
&\leq C\langle t\rangle^{-1}{E}_3(u(t))\mathcal {E}_3^{1/2}(u(t)).
\end{align}
For the second term on the right-hand side of \eqref{ssss234}, by \eqref{weight22} we have
\begin{align}\label{lioaj45r}
&\langle t\rangle^{-1}\|q'(\sigma) \nabla Z^{a}u u\nabla Z^au\|_{L^1(r\geq \langle t\rangle/2)}\nonumber\\
&\leq C\langle t\rangle^{-3/2}\|\langle t-r\rangle^{-2} \nabla Z^{a}u~ r^{1/2}u\nabla Z^au\|_{L^1(r\geq \langle t\rangle/2)}
\nonumber\\
&\leq C\langle t\rangle^{-3/2}\| \nabla Z^{a}u \|_{L^2}\|r^{1/2}u\|_{L^{\infty}}\|\nabla Z^au\|_{L^2}\nonumber\\
&\leq C\langle t\rangle^{-3/2}{E}^{3/2}_3(u(t)).
\end{align}
It follows from \eqref{ssss234}--\eqref{lioaj45r}
\begin{align}
&\sum_{\substack{|a|= 2\\a_4\leq 1}}\|e^{-q(\sigma)}q'(\sigma)\widetilde{N}(Z^{a}u, u, Z^{a}u)\|_{L^1(r\geq \langle t\rangle/2)}\nonumber\\
&~~~~~\leq C\langle t\rangle^{-1}{E}_3(u(t))\mathcal {E}_3^{1/2}(u(t))+C\langle t\rangle^{-3/2}{E}^{3/2}_3(u(t)).
\end{align}
\par
The third term on the right-hand side of \eqref{high} can be estimated by the same way as \eqref{dffggg}.
\par
For the fourth term on the right-hand side of \eqref{high}, for any $|a|\leq 2, a_4\leq 1, b+c+d=a, |b|,|c|\leq 1$, by Lemma \ref{Nulll},  \eqref{xuyao1}, \eqref{111222}, \eqref{compoi1} and noting that $\widetilde{\Omega}u=0$ we have
\begin{align}
&\big|\langle \partial_t Z^{a}u, N_d(Z^{b}u,Z^{c}u)\rangle \big|\nonumber\\
&\leq \frac{C}{r}|\partial_t Z^{a}u|\big(|\nabla \Omega Z^{b}u||\nabla Z^{c}u|+|\nabla \Omega Z^{c}u||\nabla Z^{b}u|+|\Omega Z^{b}u||\nabla^2Z^{c}u|+|\Omega Z^{c}u||\nabla^2 Z^{b}u|\big)\nonumber\\
&\leq \frac{C}{r}|\partial_t Z^{a}u|\big(|\nabla Z^{b}u||\nabla  Z^{c}u|+| Z^{b}u||\nabla^2Z^{c}u|+| Z^{c}u||\nabla^2 Z^{b}u|\big).
\end{align}
So it follows from \eqref{weight11} and \eqref{weight22} that
\begin{align}
&\sum_{\substack{|a|\leq 2\\a_4\leq 1}}\sum_{\substack{b+c+d=a\\|b|,|c|\leq 1}}\big\|e^{-q(\sigma)}\langle \partial_t Z^{a}u, N_d(Z^{b}u,Z^{c}u)\rangle \big\|_{L^1(r\geq \langle t\rangle/2)}\nonumber\\
&\leq C\langle t\rangle^{-3/2}\sum_{\substack{|a|\leq 2\\a_4\leq 1}}\sum_{\substack{b+c\leq a\\|b|,|c|\leq 1}}\|\partial_t Z^{a}u\|_{L^2}\big(\|r\nabla  Z^{b}u\|_{L^{\infty}}\|\nabla Z^{c}u\|_{L^2}+\|r^{1/2} Z^{b}u\|_{L^{\infty}}\|\nabla^2Z^{c}u\|_{L^2}\nonumber\\
&~~~~~~~~~~~~~~~~~~~~~~~~~~~~~~~~~~+\|r^{1/2} Z^{c}u\|_{L^{\infty}}\|\nabla^2 Z^{b}u\|_{L^2}\big)
\nonumber\\
&\leq C\langle t\rangle^{-3/2}{E}^{3/2}_3(u(t)).
\end{align}
\par

Hence by the above argument, we know that on the region $r\geq \langle t\rangle/2$, the right-hand side of \eqref{high}
admits an upper bound of the form
 \begin{align}
\langle t\rangle^{-3/2}E^{3/2}_{3}(u(t))+\langle t\rangle^{-1}{E}_3(u(t))\mathcal {E}_3^{1/2}(u(t)).
 \end{align}
 \par
{\bf{On the region $1/2\leq r\leq \langle t\rangle/2$.}} On the region $1/2\leq r\leq \langle t\rangle/2$, all terms on the right-hand side of \eqref{high}
 are bounded above by
\begin{align}\label{friwe}
&\sum_{\substack{|a|= 2\\ a_4\leq 1}}\|\langle t-r \rangle^{-2}\nabla Z^{a}u\nabla u \nabla Z^{a}u \|_{L^1(1/2\leq r\leq \langle t\rangle/2)}
+\sum_{\substack{|a|= 2\\a_4\leq 1}}\|\nabla Z^{a}u\partial \nabla u\nabla Z^{a}u\|_{L^1(1/2\leq r\leq \langle t\rangle/2)}
\nonumber\\
&+\sum_{\substack{|a|\leq 2\\a_4\leq 1}}\sum_{\substack{b+c\leq a\\|b|,|c|\leq 1}}\|\partial Z^{a}u \nabla Z^{b} u\nabla^2 Z^{c}u\|_{L^1(1/2\leq r\leq \langle t\rangle/2)}.
\end{align}
\par
For the first term on the right-hand side of \eqref{friwe}, for $|a|=2, a_4\leq 1$, it follows from the Sobolev embedding $H^2(\mathbb{R}^3)\hookrightarrow L^{\infty}(\mathbb{R}^3)$ that
\begin{align}\label{sameway}
&\|\langle t-r \rangle^{-2}\nabla Z^{a}u \nabla u\nabla Z^{a}u\|_{L^1(1/2\leq r\leq \langle t\rangle/2)}\nonumber\\
&\leq C\langle t\rangle^{-2}\|\nabla Z^{a}u~  \nabla u \nabla Z^{a}u\|_{L^1(1/2\leq r\leq \langle t\rangle/2)}\nonumber\\
&\leq C\langle t\rangle^{-2}\|\nabla Z^{a}u\|^2_{L^2} \| \nabla u\|_{L^{\infty}}\nonumber\\
&\leq C\langle t\rangle^{-2}E_{3}^{3/2}(u(t)) .
\end{align}
\par
For the second term on the right-hand side of \eqref{friwe}, for $|a|=2, a_4\leq 1$, by \eqref{weight44} and Proposition \ref{prop2233} we have
\begin{align}\label{similar333}
&\|\nabla Z^{a}u\partial \nabla u\nabla Z^{a}u\|_{L^1(1/2\leq r\leq \langle t\rangle/2)}\nonumber\\
&\leq C\langle t\rangle^{-1}\|\nabla Z^{a}u~ r^{1/2}\langle t-r\rangle\partial \nabla u \langle r\rangle^{-1/2}\nabla Z^{a}u\|_{L^1(1/2\leq r\leq \langle t\rangle/2)}\nonumber\\
&\leq C\langle t\rangle^{-1}\|\nabla Z^{a}u\|_{L^2} \|r^{1/2}\langle t-r\rangle\partial \nabla u\|_{L^{\infty}} \|\langle r\rangle^{-1/2}\nabla Z^{a}u\|_{L^2}\nonumber\\
&\leq C\langle t\rangle^{-1}E_{3}^{1/2}(u(t))\big(E_{2}^{1/2}(u(t))+\mathcal {X}_3(u(t))\big)  \|\langle r\rangle^{-1/2}\nabla Z^{a}u\|_{L^2}\nonumber\\
&\leq C\langle t\rangle^{-1}E_{3}(u(t)) \|\langle r\rangle^{-1/4}r^{-1/4}\nabla Z^{a}u\|_{L^2}
\nonumber\\
&\leq C\langle t\rangle^{-1}E_{3}(u(t)) \mathcal {N}_3^{1/2}(u(t)).
\end{align}
\par
As for the third term on the right-hand side of \eqref{friwe}, for $|a|\leq 2, a_4\leq 1, b+c\leq a, |b|,|c|\leq 1$, if $b_4=0$, by \eqref{weight44} and Proposition \ref{prop2233} we have
\begin{align}
&\|\partial Z^{a}u \nabla Z^{b} u\nabla^2 Z^{c}u\|_{L^1(1/2\leq r\leq \langle t\rangle/2)}\nonumber\\
&\leq C\langle t\rangle^{-1}\|\nabla^2 Z^{c}u  ~ r^{1/2}\langle t-r\rangle\nabla Z^{b} u  \langle r\rangle^{-1/2}\partial Z^{a}u\|_{L^1(1/2\leq r\leq \langle t\rangle/2)}\nonumber\\
&\leq C\langle t\rangle^{-1} \|\nabla^2 Z^{c}u\|_{L^2}\| r^{1/2}\langle t-r\rangle\nabla Z^{b} u \|_{L^{\infty}}\|\langle r\rangle^{-1/2}\partial Z^{a}u\|_{L^2}\nonumber\\
&\leq C\langle t\rangle^{-1}E_{3}^{1/2}(u(t))\big(E_{2}^{1/2}(u(t))+\mathcal {X}_3(u(t))\big)  \|\langle r\rangle^{-1/2}\partial Z^{a}u\|_{L^2}\nonumber\\
&\leq C\langle t\rangle^{-1}E_{3}(u(t)) \|\langle r\rangle^{-1/4}r^{-1/4}\partial Z^{a}u\|_{L^2}\nonumber\\
&\leq C\langle t\rangle^{-1}E_{3}(u(t)) \mathcal {N}_3^{1/2}(u(t)).
\end{align}
If $c_4=0$, it follows from \eqref{weight22} and Proposition \ref{prop2233} that
\begin{align}
&\|\partial Z^{a}u \nabla Z^{b} u\nabla^2 Z^{c}u\|_{L^1(1/2\leq r\leq \langle t\rangle/2)}\nonumber\\
&\leq C\langle t\rangle^{-1}\|\langle t-r\rangle\nabla^2 Z^{c}u  ~ r^{1/2}\nabla Z^{b} u  \langle r\rangle^{-1/2}\partial Z^{a}u\|_{L^1(1/2\leq r\leq \langle t\rangle/2)}\nonumber\\
&\leq C\langle t\rangle^{-1} \|\langle t-r\rangle\nabla^2 Z^{c}u\|_{L^2}\| r^{1/2}\nabla Z^{b} u \|_{L^{\infty}}\|\langle r\rangle^{-1/2}\partial Z^{a}u\|_{L^2}\nonumber\\
&\leq C\langle t\rangle^{-1}\mathcal {X}_3(u(t))E_{3}^{1/2}(u(t))  \|\langle r\rangle^{-1/2}\partial Z^{a}u\|_{L^2}\nonumber\\
&\leq C\langle t\rangle^{-1}E_{3}(u(t)) \|\langle r\rangle^{-1/4}r^{-1/4}\partial Z^{a}u\|_{L^2}\nonumber\\
&\leq C\langle t\rangle^{-1}E_{3}(u(t)) \mathcal {N}_3^{1/2}(u(t)).
\end{align}

\par
Hence by the above argument, we know that on the region $1/2\leq r\leq \langle t\rangle/2$, the right-hand side of \eqref{high}
  admits an upper bound of the form
 \begin{align}
 \langle t\rangle^{-2}E_{3}^{3/2}(u(t))+\langle t\rangle^{-1}E_{3}(u(t))\mathcal {N}_3^{1/2}(u(t)).
 \end{align}
 \par
{\bf{On the region $r\leq 1/2$.}} Similarly to \eqref{friwe}, on the region $r\leq 1/2$, all terms on the right-hand side of \eqref{high}
 are bounded above by
\begin{align}\label{friwe22}
&\sum_{\substack{|a|= 2\\a_4\leq 1}}\|\langle t-r \rangle^{-2}\nabla Z^{a}u\nabla u \nabla Z^{a}u \|_{L^1(r\leq 1/2)}
+\sum_{\substack{|a|= 2\\a_4\leq 1}}\|\nabla Z^{a}u\partial \nabla u\nabla Z^{a}u\|_{L^1(r\leq 1/2)}\nonumber\\
&+\sum_{\substack{|a|\leq 2\\a_4\leq 1}}\sum_{\substack{b+c\leq a\\|b|,|c|\leq 1}}\|\partial Z^{a}u \nabla Z^{b} u\nabla^2 Z^{c}u\|_{L^1(r\leq 1/2)}.
\end{align}
\par
The first term on the right-hand side of \eqref{friwe22} can be estimated by the same way as \eqref{sameway}. For $|a|=2, a_4\leq 1$, we have
\begin{align}
&\|\langle t-r \rangle^{-2}\nabla Z^{a}u \nabla u\nabla Z^{a}u\|_{L^1(r\leq 1/2)}\nonumber\\
&\leq C\langle t\rangle^{-2}\|\nabla Z^{a}u \nabla u \nabla Z^{a}u\|_{L^1(r\leq 1/2)}\nonumber\\
&\leq C\langle t\rangle^{-2}\|\nabla Z^{a}u\|^2_{L^2} \| \nabla u\|_{L^{\infty}}\nonumber\\
&\leq C\langle t\rangle^{-2}E_{3}^{3/2}(u(t)) .
\end{align}
\par
For the second term on the right-hand side of \eqref{friwe22}, for $|a|=2, a_4\leq 1$, by \eqref{weight44} and Proposition \ref{prop2233} we have
\begin{align}\label{similer556}
&\|\nabla Z^{a}u\partial \nabla u\nabla Z^{a}u\|_{L^1(r\leq 1/2)}\nonumber\\
&\leq C\langle t\rangle^{-1}\|r^{-1/4}\nabla Z^{a}u~ r^{1/2}\langle t-r\rangle\partial \nabla u~  r^{-1/4}\nabla Z^{a}u\|_{L^1(r\leq 1/2)}\nonumber\\
&\leq C\langle t\rangle^{-1}\|r^{-1/4}\nabla Z^{a}u\|^2_{L^2(r\leq 1)} \|r^{1/2}\langle t-r\rangle\partial \nabla u\|_{L^{\infty}}\nonumber\\
&\leq C\langle t\rangle^{-1}\|r^{-1/4}\nabla Z^{a}u\|^2_{L^2(r\leq 1)}\big(E_{2}^{1/2}(u(t))+\mathcal {X}_3(u(t))\big)\nonumber\\
&\leq C\langle t\rangle^{-1}E_{3}^{1/2}(u(t))\|\langle r\rangle^{-1/4}r^{-1/4}\nabla Z^{a}u\|^2_{L^2}\nonumber\\
&\leq C\langle t\rangle^{-1}E_{3}^{1/2}(u(t))\mathcal {N}_3(u(t)).
\end{align}
\par
For the third term on the right-hand side of \eqref{friwe22}, for $|a|\leq 2, a_4\leq 1, b+c\leq a,|b|,|c|\leq 1$, if $b_4=0$,  by \eqref{weight44} and Proposition \ref{prop2233} we have
\begin{align}
&\|\partial Z^{a}u \nabla Z^{b} u\nabla^2 Z^{c}u\|_{L^1(r\leq 1/2)}\nonumber\\
&\leq C\langle t\rangle^{-1}\| r^{-1/4}\partial Z^{a}u ~ r^{1/2}\langle t-r\rangle\nabla Z^{b} u ~ r^{-1/4}\nabla^2 Z^{c}u\|_{L^1(r\leq 1/2)}\nonumber\\
&\leq C\langle t\rangle^{-1}\|r^{-1/4}\partial Z^{a}u\|_{L^2(r\leq 1)} \| r^{1/2}\langle t-r\rangle\nabla Z^{b} u \|_{L^{\infty}}\|\langle r\rangle^{-1/4}\nabla^2 Z^{c}u\|_{L^2(r\leq 1)}\nonumber\\
&\leq C\langle t\rangle^{-1}\|r^{-1/4}\partial Z^{a}u\|_{L^2(r\leq 1)} \big(E_{2}^{1/2}(u(t))+\mathcal {X}_3(u(t))\big)\| r^{-1/4}\nabla^2 Z^{c}u\|_{L^2(r\leq 1)}\nonumber\\
&\leq C\langle t\rangle^{-1}E_{3}^{1/2}(u(t))\| \langle r\rangle^{-1/4}r^{-1/4}\partial Z^{a}u\|_{L^2}\| \langle r\rangle^{-1/4}r^{-1/4}\nabla^2 Z^{c}u\|_{L^2}\nonumber\\
&\leq C\langle t\rangle^{-1}E_{3}^{1/2}(u(t))\mathcal {N}_3(u(t)).
\end{align}
If $c_4=0$,  by \eqref{charu00} and Proposition \ref{prop2233} we have
\begin{align}
&\|\partial Z^{a}u \nabla Z^{b} u\nabla^2 Z^{c}u\|_{L^1(r\leq 1/2)}\nonumber\\
&\leq C\langle t\rangle^{-1}\| r^{-1/4}\partial Z^{a}u ~ r^{1/4}\nabla Z^{b} u \langle t-r\rangle\nabla^2 Z^{c}u\|_{L^1(r\leq 1/2)}\nonumber\\
&\leq C\langle t\rangle^{-1}\|r^{-1/4}\partial Z^{a}u\|_{L^2(r\leq 1)} \| r^{1/4}\nabla Z^{b} u \|_{L^{\infty}(r\leq 1)}\|\langle t-r\rangle\nabla^2 Z^{c}u\|_{L^2}\nonumber\\
&\leq C\langle t\rangle^{-1}\mathcal {X}_3(u(t))\|r^{-1/4}\partial Z^{a}u\|_{L^2(r\leq 1)} \nonumber\\
&~~\big(\| \langle r\rangle^{-1/4} r^{-1/4}\nabla^2 Z^{b}u\|_{L^2}+\| \langle r\rangle^{-1/4} r^{-1/4}\nabla Z^{b}u\|_{L^2}+\| \langle r\rangle^{-1/4} r^{-5/4}\nabla Z^{b}u\|_{L^2}\big)\nonumber\\
&\leq C\langle t\rangle^{-1}E_{3}^{1/2}(u(t))\|\langle r\rangle^{-1/4} r^{-1/4}\partial Z^{a}u\|_{L^2} \nonumber\\
&~~\big(\| \langle r\rangle^{-1/4} r^{-1/4}\nabla^2 Z^{b}u\|_{L^2}+\| \langle r\rangle^{-1/4} r^{-1/4}\nabla Z^{b}u\|_{L^2}+\| \langle r\rangle^{-1/4} r^{-5/4}\nabla Z^{b}u\|_{L^2}\big)\nonumber\\
&\leq C\langle t\rangle^{-1}E_{3}^{1/2}(u(t))\mathcal {N}_3(u(t)).
\end{align}
\par
Hence by the above argument, we know that on the region $r\leq 1/2$, the right-hand side of \eqref{high}
admits an upper bound of the form
 \begin{align}
 &\langle t\rangle^{-2}E_{3}^{3/2}(u(t))+\langle t\rangle^{-1}E^{1/2}_{3}(u(t)) \mathcal {N}_3(u(t)).
 \end{align}
 \par

 \subsection{Conclusion of the proof}
  From the above three parts argument, we have that
 \begin{align}\label{sssdddd}
 &\frac{d}{dt}\widetilde{E}_3(u(t))+\mathcal {E}_3(u(t))\nonumber\\
 &\leq C\langle t\rangle^{-1}{E}_3(u(t))\mathcal {E}_3^{1/2}(u(t))+C\langle t\rangle^{-3/2}E^{3/2}_{3}(u(t)\nonumber\\
 &+C\langle t\rangle^{-1}E_{3}(u(t))\mathcal {N}^{1/2}_3(u(t))+C\langle t\rangle^{-1}E^{1/2}_{3}(u(t)) \mathcal {N}_3(u(t))\nonumber\\
 &\leq \frac{1}{4}\mathcal {E}_3(u(t))+C\langle t\rangle^{-2}E^{2}_{3}(u(t)+C\langle t\rangle^{-3/2}E^{3/2}_{3}(u(t)\nonumber\\
 &+C\langle t\rangle^{-1}E_{3}(u(t))\mathcal {N}^{1/2}_3(u(t))+C\langle t\rangle^{-1}E^{1/2}_{3}(u(t)) \mathcal {N}_3(u(t)).
 \end{align}
The first term on the right-hand side of \eqref{sssdddd} can be absorbed to the left-hand side. So we have
 \begin{align}\label{sssdddd111}
 &\frac{d}{dt}\widetilde{E}_3(u(t))\nonumber\\
 &\leq C\langle t\rangle^{-3/2}E^{3/2}_{3}(u(t)+C\langle t\rangle^{-1}E_{3}(u(t))\mathcal {N}^{1/2}_3(u(t))+C\langle t\rangle^{-1}E^{1/2}_{3}(u(t)) \mathcal {N}_3(u(t)).
 \end{align}
 Integrating on time from $0$ to $T$ and noting \eqref{equiee}, we have that
 \begin{align}\label{sssdddd222}
&\sup_{0\leq t<T}{E}_3(u(t))\nonumber\\
&\leq C {E}_3(u(0))+ C\int_0^{T}\langle t\rangle^{-3/2}E^{3/2}_{3}(u(t))dt\nonumber\\
&+C\int_0^{T}\langle t\rangle^{-1}E_{3}(u(t)) \mathcal {N}^{1/2}_3(u(t))dt\nonumber\\
 &+C\int_0^{T}\langle t\rangle^{-1}E^{1/2}_{3}(u(t)) \mathcal {N}_3(u(t))dt.
 \end{align}
 \par
 It is obvious that
 \begin{align}\label{326}
 \int_0^{T}\langle t\rangle^{-3/2}E^{3/2}_{3}(u(t))dt\leq \int_0^{+\infty}\langle t\rangle^{-3/2}dt\sup_{0\leq t<T}E^{3/2}_3(u(t))\leq C\sup_{0\leq t<T}E^{3/2}_3(u(t)).
 \end{align}
 \par
 In order to treat the last two terms on the right-hand side of \eqref{sssdddd222}, we will employ the KSS type energy $\mathcal {M}_3(u(t))$ and the dyadic decomposition technique in time (see page 363 of \cite{MR1979951}).
 \par
 Without loss of generality, we can assume $T>1$.
 It follows from the H\"{o}lder inequality and Proposition \ref{prop2233kss} that
 \begin{align}\label{liaotian}
&\int_0^{1}\langle t\rangle^{-1}E_{3}(u(t)) \mathcal {N}^{1/2}_3(u(t))dt\nonumber\\
&\leq C\sup_{0\leq t\leq 1}E_{3}(u(t)) \Big(\int_0^1\mathcal {N}_3(u(t))dt\Big)^{1/2}\nonumber\\
&\leq C\sup_{0\leq t\leq 1}E_{3}(u(t)) \mathcal {M}^{1/2}_3(u(1))\nonumber\\
&\leq C\sup_{0\leq t< T}E^{3/2}_{3}(u(t)).
 \end{align}
 Take an integer $N$ such that $2^{N}<T\leq 2^{N+1}$.  By the H\"{o}lder inequality and Proposition \ref{prop2233kss}, we have
 \begin{align}\label{simil68jn}
 &\int_1^{T}\langle t\rangle^{-1}E_{3}(u(t))  \mathcal {N}^{1/2}_3(u(t))dt\nonumber\\
 &\leq \sum_{k=0}^{N}2^{-k}\int_{2^{k}}^{2^{k+1}} \mathcal {N}^{1/2}_3(u(t))dt\sup_{0\leq t<T}E_3(u(t))\nonumber\\
 &\leq \sum_{k=0}^{N}2^{-k/2} \big(\int_{2^{k}}^{2^{k+1}} \mathcal {N}_3(u(t))dt\big)^{1/2}\sup_{0\leq t<T}E_3(u(t))\nonumber\\
 &\leq C\sum_{k=0}^{N}2^{-k/2}\mathcal {M}^{1/2}_3(u(2^{k+1}))\sup_{0\leq t<T}E_3(u(t))
 \nonumber\\
 &\leq C\sum_{k=0}^{N}2^{-k/2}2^{\delta k}(\log(2+2^{k+1}))^{1/2}\sup_{0\leq t<T}E^{3/2}_3(u(t)).
 \end{align}
 Noting that $\delta$ is sufficiently small, we have
 \begin{align}\label{123dwwwd}
 \sum_{k=0}^{N}2^{-k/2}2^{\delta k}(\log(2+2^{k+1}))^{1/2}\leq  \sum_{k=0}^{+\infty}2^{-k/2}2^{\delta k}(\log(2+2^{k+1}))^{1/2}\leq C.
 \end{align}
 The combination of \eqref{liaotian}--\eqref{123dwwwd} gives
 \begin{align}\label{123ddggffff}
 \int_0^{T}\langle t\rangle^{-1}E_{3}(u(t))   \mathcal {N}^{1/2}_3(u(t))dt
 \leq C\sup_{0\leq t<T}E^{3/2}_3(u(t)).
 \end{align}
 \par
By Proposition \ref{prop2233kss}, we have
 \begin{align}\label{new11weee}
&\int_0^{1}\langle t\rangle^{-1}E^{1/2}_{3}(u(t))   \mathcal {N}_3(u(t))dt\nonumber\\
&\leq C\sup_{0\leq t\leq 1}E^{1/2}_{3}(u(t))  \int_0^{1} \mathcal {N}_3(u(t))dt\nonumber\\
&\leq C\sup_{0\leq t\leq 1}E^{1/2}_{3}(u(t)) \mathcal {M}_3(u(1))\nonumber\\
&\leq C\sup_{0\leq t< T}E^{3/2}_{3}(u(t)).
 \end{align}
 Similarly to \eqref{simil68jn}, it follows from Proposition \ref{prop2233kss} that
\begin{align}
&\int_1^{T}\langle t\rangle^{-1}E^{1/2}_{3}(u(t))  \mathcal {N}_3(u(t))dtdt\nonumber\\
&\leq \sum_{k=0}^{N}2^{-k}\int_{2^k}^{2^{k+1}}  \mathcal {N}_3(u(t))dt\sup_{0\leq t<T}E^{1/2}_{3}(u(t))
\nonumber\\
&\leq C\sum_{k=0}^{N}2^{-k}\mathcal {M}_3(u(2^{k+1}))\sup_{0\leq t<T}E^{1/2}_{3}(u(t))
\nonumber\\
&\leq C\sum_{k=0}^{N}2^{-k}2^{2\delta k}(\log(2+2^{k+1}))\sup_{0\leq t<T}E^{3/2}_{3}(u(t)).
\end{align}
Noting that $\delta$ is sufficiently small, we have
 \begin{align}\label{123dwwwdfffdd}
\sum_{k=0}^{N}2^{-k}2^{2\delta k}(\log(2+2^{k+1}))\leq  \sum_{k=0}^{+\infty}2^{-k}2^{2\delta k}(\log(2+2^{k+1}))\leq C.
 \end{align}
 The combination of \eqref{new11weee}--\eqref{123dwwwdfffdd} gives
 \begin{align}\label{1dgttgffffwwww}
 \int_0^{T}\langle t\rangle^{-1}E^{1/2}_{3}(u(t))    \mathcal {N}_3(u(t))dt
 \leq C\sup_{0\leq t<T}E^{3/2}_3(u(t)).
 \end{align}
 \par
 \par
 It follows from \eqref{sssdddd222}, \eqref{326}, \eqref{123ddggffff} and \eqref{1dgttgffffwwww} that
  \begin{align}\label{sssddddzuizhong}
\sup_{0\leq t<T}{E}_3(u(t))\leq C {E}_3(u(0))+C_2\sup_{0\leq t<T}E^{3/2}_3(u(t))\leq C_1\varepsilon^2+8C_2A^3\varepsilon^3.
 \end{align}
 Take $A^2=2C_1$ and $\varepsilon_0$ so small that
 \begin{align}
 16C_2A\varepsilon_0\leq 1.
 \end{align}
Then for any $\varepsilon$ with $0<\varepsilon\leq \varepsilon_0$ we have
\begin{align}
\sup_{0\leq t< T}E^{1/2}_{k}(u(t))\leq A\varepsilon,
\end{align}
which completes the proof of Proposition \ref{mainthm33333}.

\appendix
  \renewcommand\thesection{\appendixname~\Alph{section}}
\renewcommand\theequation{\Alph{section}.\arabic{equation}}
\section{Proof of Lemma \ref{thmKSS00}} \label{appen}
The proof of Lemma 2.4 repeats essentially the same arguments
as in the preceding papers \cite{MR2128434}, \cite{MR2217314}, \cite{MR2919103},
\cite{MR3650329}, and \cite{almost}.
It therefore suffices to give only the sketch of the proof.
In the following, we use the notation $M:=f(r)\partial_r+(1/r)f(r)$
for a function $f(r)$ to be given later.
Proceeding as in (2.12)-(2.15) of \cite{almost},
we can establish the differential identity
for ${\mathbb R}^3$-valued functions $u=(u^1,u^2,u^3)$
\begin{equation}
\langle
Mu,\Box u
\rangle
=
\partial_t e
+
\nabla\cdot{\hat p}
+
{\hat q},
\end{equation}
where, by $\langle\cdot,\cdot\rangle$
we mean the inner product in ${\mathbb R}^3$,
and
\begin{align}
e&=f(r)(\partial_t u^i)(\partial_r u^i+\frac{1}{r}u^i),\\
{\hat p}&=
\frac{1}{2}f(r)(|\nabla u|^2-|\partial_t u|^2)\omega
-
f(r)(\partial_r u^i+\frac{1}{r}u^i)\nabla u^i\nonumber\\
&+
\frac{rf'(r)-f(r)}{2r^2}|u|^2\omega,\\
{\hat q}&=\frac{f'(r)}{2}|\partial_t u|^2
+\frac{f'(r)}{2}|\partial_r u|^2\nonumber\\
&+
\left(
\frac{f(r)}{r}-\frac{f'(r)}{2}
\right)
|\nabla_\omega u|^2
-
\frac{1}{2}
\left(
\Delta\frac{f(r)}{r}
\right)
|u|^2,
\end{align}
and $\omega\in S^2$.
Recall that repeated indices are summed, regardless of their
position up or down.
Also, as in (2.21) of \cite{almost}, we get
\begin{equation}
\langle
Mu,Hu
\rangle
=
\nabla\cdot{\tilde p}
+
{\tilde q},
\end{equation}
where
\begin{align}
{\tilde p}_l
&=
f(r)\omega_kh_{lm}^{ij}(\partial_k u^i)(\partial_m u^j)
-\frac{1}{2}f(r)\omega_lh_{km}^{ij}
(\partial_k u^i)(\partial_m u^j)\nonumber\\
&+
\frac{1}{r}f(r)h_{lm}^{ij}u^i\partial_m u^j,\quad l=1,2,3,\\
{\tilde q}&=
-\frac{rf'(r)-f(r)}{r}
\omega_k\omega_l
h_{lm}^{ij}(\partial_k u^i)(\partial_m u^j)
+
\frac{1}{2}f'(r)h_{lm}^{ij}(\partial_l u^i)(\partial_m u^j)\nonumber\\
&+
\frac{1}{2}f(r)\omega_k
(\partial_k h_{lm}^{ij})(\partial_l u^i)(\partial_m u^j)
-
\frac{1}{r}
f(r)h_{lm}^{ij}(\partial_l u^i)(\partial_m u^j)\nonumber\\
&-\frac{rf'(r)-f(r)}{r^2}\omega_lh_{lm}^{ij}u^i\partial_m u^j.
\end{align}
First, we take as in \cite{MR2919103} and \cite{almost}
\begin{equation}\label{A8}
f(r)
=
\left(
\frac{r}{1+r}
\right)^{1/2}.
\end{equation}
Proceeding as in (2.37)-(2.42) of \cite{MR2919103},
we can obtain
\begin{align}\label{A9}
&
\|r^{-1/4}\partial u\|_{L^2(0,t;L^2(|x|\leq 1))}^2
+
\|r^{-5/4}u\|_{L^2(0,t;L^2(|x|\leq 1))}^2\nonumber\\
&
\leq
C\sup_{0\leq \tau \leq t}\int_{{\mathbb R}^3}|e|dx
+
C\int_0^t\!\!\!\int_{{\mathbb R}^3}|{\tilde q}|dxd\tau
+
C\int_0^t\!\!\!\int_{{\mathbb R}^3}
|
\langle
Mu,\Box_h u
\rangle
|dxd\tau\nonumber\\
&
\leq
C\|\partial u(0,\cdot)\|_{L^2({\mathbb R}^3)}^2
+
C\int_0^t\!\!\!\int_{{\mathbb R}^3}
|\nabla h||\nabla u|^2
dxd\tau\nonumber\\
&
+
C\int_0^t\!\!\!\int_{{\mathbb R}^3}
\frac{|h|}{r^{1/2}(1+r)^{1/2}}
\left(
\frac{|u|}{r}+|\nabla u|
\right)
|\nabla u|
dxd\tau\nonumber\\
&+
C\int_0^t\!\!\!\int_{{\mathbb R}^3}
\left(
|\nabla u|
+
\frac{|u|}{r^{1/2}(1+r)^{1/2}}
\right)
|\Box_h u|
dxd\tau.
\end{align}
Here, we have also used the fact that
thanks to the symmetry and smallness conditions
\eqref{duichengxing678}-\eqref{smallness},
 the standard energy inequality
\begin{align}\label{A10}
&\sup_{0\leq \tau\leq t}\|\partial u(\tau,\cdot)\|^2_{L^2(\mathbb{R}^{3})}\nonumber\\
&\leq C\|\partial u(0,\cdot)\|^2_{L^2(\mathbb{R}^{3})}+C\int_0^t\!\!\!\int_{{\mathbb R}^3}|\partial_th||\nabla u|^2dxd\tau+C\int_0^t\!\!\!\int_{{\mathbb R}^3}| \partial_tu||\Box_hu |dxd\tau
\end{align}
holds.
We note that owing to the Hardy inequality,
\eqref{A10} also yields
\begin{align}\label{A11}
&
\|\langle r\rangle^{-1/2}
\partial u\|_{L^2(0,t;L^2(|x|\geq t))}^2
+
\|\langle r\rangle^{-3/2}
u\|_{L^2(0,t;L^2(|x|\geq t))}^2\nonumber\\
&\leq C\sup_{0\leq \tau\leq t}\|\partial u(\tau,\cdot)\|^2_{L^2(\mathbb{R}^{3})}\nonumber\\
&\leq C\|\partial u(0,\cdot)\|^2_{L^2(\mathbb{R}^{3})}+C\int_0^t\!\!\!\int_{{\mathbb R}^3}|\partial_th||\nabla u|^2dxd\tau+C\int_0^t\!\!\!\int_{{\mathbb R}^3}| \partial_tu||\Box_hu |dxd\tau.
\end{align}
It remains to consider the weighted $L^2$ estimate
over $[0,t]\times\{x\in{\mathbb R}^3:1<|x|<t\}$ with $t>1$.
We go back to \eqref{A8}, and take as in \cite{MR2128434},
\cite{MR2217314}, and \cite{MR3650329}
\begin{equation}
f(r)
=
\frac{r}{r+\rho},\quad \rho\geq 1.
\end{equation}
Proceeding as in (2.52), (2.53), and (2.57) of \cite{MR3650329},
we can obtain
\begin{align}\label{A13}
&
\int_0^t\!\!\!\int_{{\mathbb R}^3}
\left(
\frac{\rho}{(r+\rho)^2}
|\partial_\tau u|^2
+
\frac{\rho}{(r+\rho)^2}|\partial_r u|^2
\right.\nonumber\\
&
\left.
\hspace{1.5cm}
+
\frac{2r+\rho}{(r+\rho)^2}|\nabla_\omega u|^2
+
\frac{\rho}{r(r+\rho)^3}|u|^2
\right)
dxd\tau\nonumber\\
&\leq
C
\|
\partial u(0,\cdot)
\|_{L^2({\mathbb R}^3)}^2\nonumber\\
&
+
C
\int_0^t\!\!\!\int_{{\mathbb R}^3}
\left(
|\partial h|
+
\frac{|h|}{r+1}
\right)
|\nabla u|
\left(
|\nabla u|
+
\frac{|u|}{r+1}
\right)
dxd\tau\nonumber\\
&+
C
\int_0^t\!\!\!\int_{{\mathbb R}^3}
\left(
|\nabla u|
+
\frac{|u|}{r+1}
\right)
|\Box_h u|
dxd\tau.
\end{align}
Here all the constants $C$ on the right-hand side above
are independent of $\rho\geq 1$.
It is now a routine practice (see, e.g, (2.70)-(2.71) of \cite{MR3650329})
to obtain from \eqref{A13}
\begin{align}\label{A14}
&
\bigl(
\log(2+t)
\bigr)^{-1}
\left(
\|\langle r\rangle^{-1/2}
\partial u\|_{L^2(0,t;L^2(1<|x|<t))}^2
+
\|\langle r\rangle^{-3/2}
u\|_{L^2(0,t;L^2(1<|x|<t))}^2
\right)\nonumber\\
&
\leq
C
\|
\partial u(0,\cdot)
\|_{L^2({\mathbb R}^3)}^2\nonumber\\
&
+
C\int_0^t\!\!\!\int_{{\mathbb R}^3}
\left(
|\partial h|
+
\frac{|h|}{r+1}
\right)
|\nabla u|
\left(
|\nabla u|
+
\frac{|u|}{r+1}
\right)
dxd\tau\nonumber\\
&+
C\int_0^t\!\!\!\int_{{\mathbb R}^3}
\left(
|\nabla u|
+
\frac{|u|}{r+1}
\right)
|\Box_h u|
dxd\tau.
\end{align}
Combining \eqref{A9}, \eqref{A11}, and \eqref{A14},
we have finished the proof of Lemma \ref{thmKSS00}.

 \section*{Acknowledgements}
The first author was supported in part by
JSPS KAKENHI Grant Number JP15K04955.
The second author was supported by Shanghai Sailing Program (No.17YF1400700) and the Fundamental Research Funds for the Central Universities
(No.17D110913).


\begin{thebibliography}{10}

\bibitem{Agemi00}
R.~Agemi, \href{http://dx.doi.org/10.1007/s002220000084}{Global existence of
  nonlinear elastic waves}, Invent. Math. 142 (2000) 225--250.

\bibitem{Alinhac01}
S.~Alinhac, \href{http://dx.doi.org/10.1007/s002220100165}{The null condition
  for quasilinear wave equations in two space dimensions {I}}, Invent. Math.
  145 (2001) 597--618.

\bibitem{MR2666888}
S.~Alinhac, \href{http://dx.doi.org/10.1017/CBO9781139107198}{\emph{Geometric
  analysis of hyperbolic differential equations: an introduction}},
  \emph{London Mathematical Society Lecture Note Series}, vol. 374, Cambridge
  University Press, Cambridge, 2010.

\bibitem{Christodoulou86}
D.~Christodoulou, \href{http://dx.doi.org/10.1002/cpa.3160390205}{Global
  solutions of nonlinear hyperbolic equations for small initial data}, Comm.
  Pure Appl. Math. 39 (1986) 267--282.

\bibitem{MR936420}
P.~G. Ciarlet, \emph{Mathematical elasticity. {V}ol. {I}: Three-dimensional
  elasticity}, \emph{Studies in Mathematics and its Applications}, vol.~20,
  North-Holland Publishing Co., Amsterdam, 1988.

\bibitem{MR3583356}
B.~Ettinger, H.~Lindblad,
  \href{http://dx.doi.org/10.4007/annals.2017.185.1.6}{A sharp counterexample
  to local existence of low regularity solutions to {E}instein equations in
  wave coordinates}, Ann. of Math. (2) 185 (2017) 311--330.

\bibitem{MR3286047}
Z.~Guo, Y.~Wang, \href{https://doi.org/10.1007/s11854-014-0025-6}{Improved
  {S}trichartz estimates for a class of dispersive equations in the radial case
  and their applications to nonlinear {S}chr\"odinger and wave equations}, J.
  Anal. Math. 124 (2014) 1--38.

\bibitem{Gurtin81}
M.~E. Gurtin, \emph{Topics in finite elasticity}, \emph{CBMS-NSF Regional
  Conference Series in Applied Mathematics}, vol.~35, Society for Industrial
  and Applied Mathematics (SIAM), Philadelphia, Pa., 1981.

\bibitem{MR2053322}
K.~Hidano, \href{http://projecteuclid.org/euclid.tmj/1113246554}{An elementary
  proof of global or almost global existence for quasi-linear wave equations},
  Tohoku Math. J. (2) 56 (2004) 271--287.

\bibitem{MR2569550}
K.~Hidano, \href{https://doi.org/10.4171/RMI/579}{Small solutions to
  semi-linear wave equations with radial data of critical regularity}, Rev.
  Mat. Iberoam. 25 (2009) 693--708.

\bibitem{Hidano}
K.~Hidano, \href{https://arxiv.org/abs/1610.04824}{Regularity and lifespan of
  small solutions to systems of quasi-linear wave equations with multiple
  speeds, {I}: almost global existence}, arXiv:1610.04824v2. Accepted for
  publication in RIMS K\^{o}ky\^{u}roku Bessatsu. (2016).

\bibitem{Hidano17}
K.~Hidano, J.~Jiang, S.~Lee, C.~Wang,
  \href{https://arxiv.org/abs/1610.04824}{Weighted fractional chain rule and
  nonlinear wave equations with minimal regularity}, arXiv:1605.06748v2
  (2017).

\bibitem{MR2919103}
K.~Hidano, C.~Wang, K.~Yokoyama,
  \href{http://projecteuclid.org/euclid.ade/1355703087}{On almost global
  existence and local well posedness for some 3-{D} quasi-linear wave
  equations}, Adv. Differential Equations 17 (2012) 267--306.

\bibitem{MR2262091}
K.~Hidano, K.~Yokoyama,
  \href{http://projecteuclid.org/euclid.die/1356050327}{Space-time
  {$L^2$}-estimates and life span of the {K}lainerman-{M}achedon radial
  solutions to some semi-linear wave equations}, Differential Integral
  Equations 19 (2006) 961--980.

\bibitem{almost}
K.~Hidano, D.~Zha, \href{https://arxiv.org/abs/1710.05180}{Space-time ${L}^2$
  estimates, regularity and almost global existence for elastic waves},
  arXiv:1710.05180  (2017).

\bibitem{MR1466700}
L.~H\"ormander, \emph{Lectures on nonlinear hyperbolic differential equations},
  \emph{Math\'ematiques \& Applications (Berlin) [Mathematics \&
  Applications]}, vol.~26, Springer-Verlag, Berlin, 1997.

\bibitem{MR0420024}
T.~J.~R. Hughes, T.~Kato, J.~E. Marsden,
  \href{http://dx.doi.org/10.1007/BF00251584}{Well-posed quasi-linear
  second-order hyperbolic systems with applications to nonlinear elastodynamics
  and general relativity}, Arch. Rational Mech. Anal. 63 (1976) 273--294.

\bibitem{MR2911108}
J.-C. Jiang, C.~Wang, X.~Yu,
  \href{https://doi.org/10.3934/cpaa.2012.11.1723}{Generalized and weighted
  {S}trichartz estimates}, Commun. Pure Appl. Anal. 11 (2012) 1723--1752.

\bibitem{MR1774100}
S.~Jiang, R.~Racke, \emph{Evolution equations in thermoelasticity},
  \emph{Chapman \& Hall/CRC Monographs and Surveys in Pure and Applied
  Mathematics}, vol. 112, Chapman \& Hall/CRC, Boca Raton, FL, 2000.

\bibitem{John84}
F.~John, \href{http://dx.doi.org/10.1007/3-540-12916-2_58}{Formation of
  singularities in elastic waves}, in \emph{Trends and applications of pure
  mathematics to mechanics ({P}alaiseau, 1983)}, \emph{Lecture Notes in Phys.},
  vol. 195, Springer, Berlin, 1984, 194--210.

\bibitem{John88}
F.~John, \href{http://dx.doi.org/10.1002/cpa.3160410507}{Almost global
  existence of elastic waves of finite amplitude arising from small initial
  disturbances}, Comm. Pure Appl. Math. 41 (1988) 615--666.

\bibitem{MR1945285}
M.~Keel, H.~F. Smith, C.~D. Sogge,
  \href{http://dx.doi.org/10.1007/BF02868477}{Almost global existence for some
  semilinear wave equations}, J. Anal. Math. 87 (2002) 265--279.

\bibitem{MR2015331}
M.~Keel, H.~F. Smith, C.~D. Sogge,
  \href{http://dx.doi.org/10.1090/S0894-0347-03-00443-0}{Almost global
  existence for quasilinear wave equations in three space dimensions}, J. Amer.
  Math. Soc. 17 (2004) 109--153 (electronic).

\bibitem{Klainerman86}
S.~Klainerman, The null condition and global existence to nonlinear wave
  equations, in \emph{Nonlinear systems of partial differential equations in
  applied mathematics, {P}art 1 ({S}anta {F}e, {N}.{M}., 1984)}, \emph{Lectures
  in Appl. Math.}, vol.~23, Amer. Math. Soc., Providence, RI, 1986, 293--326.

\bibitem{Klainerman98}
S.~Klainerman,
  \href{http://dx.doi.org/10.1002/(SICI)1097-0312(199809/10)51:9/10<991::AID-CPA3>3.3.CO;2-0}{On
  the work and legacy of {F}ritz {J}ohn, 1934--1991}, Comm. Pure Appl. Math. 51
  (1998) 991--1017.

\bibitem{MR1231427}
S.~Klainerman, M.~Machedon,
  \href{http://dx.doi.org/10.1002/cpa.3160460902}{Space-time estimates for null
  forms and the local existence theorem}, Comm. Pure Appl. Math. 46 (1993)
  1221--1268.

\bibitem{MR2180400}
S.~Klainerman, I.~Rodnianski,
  \href{https://doi.org/10.4007/annals.2005.161.1143}{Rough solutions of the
  {E}instein-vacuum equations}, Ann. of Math. (2) 161 (2005) 1143--1193.

\bibitem{Klainerman96}
S.~Klainerman, T.~C. Sideris,
  \href{http://dx.doi.org/10.1002/(SICI)1097-0312(199603)49:3<307::AID-CPA4>3.0.CO;2-H}{On
  almost global existence for nonrelativistic wave equations in {$3$}{D}},
  Comm. Pure Appl. Math. 49 (1996) 307--321.

\bibitem{MR3014806}
H.~Kubo, \href{http://dx.doi.org/10.1007/978-3-0348-0454-7_10}{Lower bounds for
  the lifespan of solutions to nonlinear wave equations in elasticity}, in
  \emph{Evolution equations of hyperbolic and {S}chr\"odinger type},
  \emph{Progr. Math.}, vol. 301, Birkh\"auser/Springer Basel AG, Basel, 2012,
  187--212.

\bibitem{MR1375301}
H.~Lindblad,
  \href{http://muse.jhu.edu/journals/american_journal_of_mathematics/v118/118.1lindblad.pdf}{Counterexamples
  to local existence for semi-linear wave equations}, Amer. J. Math. 118 (1996)
  1--16.

\bibitem{MR1666844}
H.~Lindblad, \href{http://dx.doi.org/10.4310/MRL.1998.v5.n5.a5}{Counterexamples
  to local existence for quasilinear wave equations}, Math. Res. Lett. 5 (1998)
  605--622.

\bibitem{MR1335386}
H.~Lindblad, C.~D. Sogge, \href{https://doi.org/10.1006/jfan.1995.1075}{On
  existence and scattering with minimal regularity for semilinear wave
  equations}, J. Funct. Anal. 130 (1995) 357--426.

\bibitem{MR2108356}
S.~Machihara, M.~Nakamura, K.~Nakanishi, T.~Ozawa,
  \href{https://doi.org/10.1016/j.jfa.2004.07.005}{Endpoint {S}trichartz
  estimates and global solutions for the nonlinear {D}irac equation}, J. Funct.
  Anal. 219 (2005) 1--20.

\bibitem{MR2249996}
J.~Metcalfe, \href{http://dx.doi.org/10.1155/IMRN/2006/69826}{Elastic waves in
  exterior domains. {I}. {A}lmost global existence}, Int. Math. Res. Not.
  (2006) Art. ID 69826, 41.

\bibitem{MR2217314}
J.~Metcalfe, C.~D. Sogge, \href{http://dx.doi.org/10.1137/050627149}{Long-time
  existence of quasilinear wave equations exterior to star-shaped obstacles via
  energy methods}, SIAM J. Math. Anal. 38 (2006) 188--209.

\bibitem{MR2344575}
J.~Metcalfe, B.~Thomases, \href{http://dx.doi.org/10.1093/imrn/rnm034}{Elastic
  waves in exterior domains. {II}. {G}lobal existence with a null structure},
  Int. Math. Res. Not. IMRN  (2007) Art. ID rnm034, 43.

\bibitem{MR2971205}
E.~Y. Ovcharov, \href{https://doi.org/10.1080/03605302.2011.632047}{Radial
  {S}trichartz estimates with application to the 2-{D} {D}irac-{K}lein-{G}ordon
  system}, Comm. Partial Differential Equations 37 (2012) 1754--1788.

\bibitem{MR1211729}
G.~Ponce, T.~C. Sideris, \href{https://doi.org/10.1080/03605309308820925}{Local
  regularity of nonlinear wave equations in three space dimensions}, Comm.
  Partial Differential Equations 18 (1993) 169--177.

\bibitem{Sideris96}
T.~C. Sideris, \href{http://dx.doi.org/10.1007/s002220050030}{The null
  condition and global existence of nonlinear elastic waves}, Invent. Math. 123
  (1996) 323--342.

\bibitem{Sideris00}
T.~C. Sideris, \href{http://dx.doi.org/10.2307/121050}{Nonresonance and global
  existence of prestressed nonlinear elastic waves}, Ann. of Math. (2) 151
  (2000) 849--874.

\bibitem{Sideris01}
T.~C. Sideris, S.-Y. Tu,
  \href{http://dx.doi.org/10.1137/S0036141000378966}{Global existence for
  systems of nonlinear wave equations in 3{D} with multiple speeds}, SIAM J.
  Math. Anal. 33 (2001) 477--488 (electronic).

\bibitem{MR2178963}
H.~F. Smith, D.~Tataru,
  \href{http://dx.doi.org/10.4007/annals.2005.162.291}{Sharp local
  well-posedness results for the nonlinear wave equation}, Ann. of Math. (2)
  162 (2005) 291--366.

\bibitem{MR1979951}
C.~D. Sogge, \href{http://dx.doi.org/10.1090/conm/320/05618}{Global existence
  for nonlinear wave equations with multiple speeds}, in \emph{Harmonic
  analysis at {M}ount {H}olyoke ({S}outh {H}adley, {MA}, 2001)}, \emph{Contemp.
  Math.}, vol. 320, Amer. Math. Soc., Providence, RI, 2003, 353--366.

\bibitem{MR2128434}
J.~Sterbenz, \href{http://dx.doi.org/10.1155/IMRN.2005.187}{Angular regularity
  and {S}trichartz estimates for the wave equation, with an appendix by {I}gor
  {R}odnianski}, Int. Math. Res. Not.  (2005) 187--231.

\bibitem{MR3656947}
Q.~Wang, \href{https://doi.org/10.1007/s40818-016-0013-5}{A geometric approach
  for sharp local well-posedness of quasilinear wave equations}, Ann. PDE 3
  (2017) Art. 12, 108.

\bibitem{zha3}
D.~Zha, On nonlinear elastic waves in the radial symmetry in 2-{D}: exterior
  problem, preprint  (2016).

\bibitem{zhaelastic}
D.~Zha,
  \href{https://aimsciences.org/journals/displayArticlesnew.jsp?paperID=12342}{Remarks
  on nonlinear elastic waves in the radial symmetry in 2-{D}}, Discrete and
  Continuous Dynamical Systems--Series A 36 (2016) 4051--4062.

\bibitem{MR3549405}
D.~Zha, \href{http://dx.doi.org/10.1002/mma.3876}{Some remarks on quasilinear
  wave equations with null condition in 3-{D}}, Math. Methods Appl. Sci. 39
  (2016) 4484--4495.

\bibitem{MR3650329}
D.~Zha, \href{http://dx.doi.org/10.1016/j.jde.2017.04.014}{Space--time {$L^2$}
  estimates for elastic waves and applications}, J. Differential Equations 263
  (2017) 1947--1965.

\bibitem{MR2482537}
Y.~Zhou, Z.~Lei, \href{https://projecteuclid.org/euclid.ade/1355867360}{Global
  low regularity solutions of quasi-linear wave equations}, Adv. Differential
  Equations 13 (2008) 55--104.

\end{thebibliography}
\end{document}